\documentclass[11pt]{article}
\usepackage{amssymb}
\usepackage{amsthm}
\usepackage{mathrsfs}
\usepackage{mathtools}
\usepackage{esint}
 \newtheorem{thm}{Theorem}[section]
 \newtheorem{cor}[thm]{Corollary}
 \newtheorem{lem}[thm]{Lemma}
 \newtheorem{prop}[thm]{Proposition}
 \theoremstyle{definition}
 \newtheorem{defn}[thm]{Definition}
 \theoremstyle{remark}
 \newtheorem{rem}[thm]{Remark}
 \newtheorem*{ex}{Examples}
 \numberwithin{equation}{section}
 
 \def\Xint#1{\mathchoice
 	{\XXint\displaystyle\textstyle{#1}}%
 	{\XXint\textstyle\scriptstyle{#1}}%
 	{\XXint\scriptstyle\scriptscriptstyle{#1}}%
 	{\XXint\scriptscriptstyle\scriptscriptstyle{#1}}%
 	\!\int}
 \def\XXint#1#2#3{{\setbox0=\hbox{$#1{#2#3}{\int}$}
 		\vcenter{\hbox{$#2#3$}}\kern-.5\wd0}}
 
 \def\dashint{\Xint-}
 \usepackage{color}

\begin{document}

%
%
%
%
%
%
%
%
%

	
\title{(Two-scale) $W^{1}L^{\Phi}$-gradient Young measures and homogenization of integral functionals in Orlicz-Sobolev spaces}
\author{Joel Fotso Tachago,\footnote{University of Bamenda, 	Higher Teachers Training College, Department of Mathematics
			P.O. Box 39
			Bambili, Cameroon, email:fotsotachago@yahoo.fr} \, Hubert Nnang, 	\footnote{University of Yaounde I
			Higher Teachers Training College, Department of Mathematics
			P.O. Box 47
			Yaounde, Cameroon email: hnnang@yahoo.fr} \\
	 Franck Tchinda,\footnote{University of Maroua, Department of Mathematics and Computer Science, P.O. box 814, Maroua, Cameroon email:takougoumfranckarnold@gmail.com} \, 	and Elvira
		Zappale\footnote{Department of Basic and Applied Sciences for Engineering, Sapienza-  University of Rome, via
			A. Scarpa,  16 (00161) Roma (SA), e-mail:
			elvira.zappale1@uniroma1.it}}\date{ } \maketitle
	\begin{abstract}

		(Two-scale) gradient Young measures in Orlicz-Sobolev setting are introduced and characterized providing also an integral representation formula for non convex energies arising in homogenization problems with nonstandard growth. 
		
	\end{abstract}











 

\maketitle


\section{Introduction} \label{sect1} 

\color{black} 

The homogenization of periodic structures has been subject of interest in the last 50 years in order to deal with problems involving differential equations and systems, nonlocal models, supremal integral energies. Many techniques have been developed: we recall the basic works [DGS], [BLP], [MT],
then the asymptotic behaviour of these materials with fine distributed periodic heterogeneities has been carried out generating many results in different
settings.
Here we focus on  problems concerning vector-valued configurations whose equilibrium states can be detected studying minimizers of nonconvex energies.

In particular we consider a family of functionals of the type 
\begin{equation}\label{intro1}
	\mathcal{F}_{\varepsilon}(u) := \int_{\Omega} f\left(x,\langle\frac{x}{\varepsilon}\rangle, \nabla u(x)\right) dx,
\end{equation}  
as $\varepsilon\to 0$. Here $\Omega$ (bounded open subset of $\mathbb{R}^{N}$) is the reference configuration of a nonlinear elastic body with periodic microstructure and whose heterogeneities scale like a small parameter $\varepsilon>0$. 
\par In this paper, we are interested in the case where the functional \eqref{intro1} is defined on the Orlicz-Sobolev spaces. The extension to the Orlicz-Sobolev's spaces is motived by the fact that  there exist problems whose solution must naturally belong not to the classical Sobolev spaces (see, e.g. \cite{mignon2}). Thus, the function $u$ in \eqref{intro1} belongs to $W^{1}L^{\Phi}(\Omega; \mathbb{R}^{d})$, with $\Phi$ a Young function of class $\Delta_{2}$, stands for a deformation and $f: \Omega\times Y\times \mathbb{R}^{d\times N} \to [0,+\infty)$, with $Y:= (0,1)^{N}$, is the stored energy density of this body that is assumed to satisfy nonstandard $\Phi$-coercivity and $\Phi$-growth conditions (see \eqref{c1intro2}), with integers $d\geq 1$ and $N\geq 1$. The presence of the term $\langle\frac{x}{\varepsilon}\rangle$ (fractional part of the vector $\dfrac{x}{\varepsilon}$ componentwise) takes into account the periodic microstructure of the body, with the integrand $f(x,\cdot,\xi)$ in \eqref{intro1} periodic. The macroscopic  description of this material may be understood by an asymptotic as $\varepsilon \to 0$, computed as the $\Gamma$-limit of \eqref{intro1} with respect to the $L^{\Phi}(\Omega; \mathbb{R}^{d})$-topology, (equivalently $W^1L^{\Phi}$- weak  if $\Omega$ is, for instance, Lipschitz). 

Since the method employed in \cite{tacha2, tacha3, tacha6, FTGNZ} to compute the $\Gamma$-limit of functionals of the type \eqref{intro1}, both under convexity and without convexity assumptions on $f(x, \langle \frac{x}{\varepsilon}\rangle, \cdot)$, has been two-scale convergence (introduced in \cite{nguet1, allair1} in the Sobolev setting and later extended to Sobolev-Orlicz spaces in \cite{tacha1}) (or periodic unfolding method introduced in \cite{CDG1, CDG2, CDGbook}, applied to homogenization of integral functionals in \cite{CDDEA1, CDDEA2} and later extended to the Orlicz framework in \cite{tacha5, FTGNZ}), we want to relate these tools to the one of Young measures following the ideas in \cite{pedregal1,pedregal2,pedregal3}.
Indeed, Young measures are an important tool for studying the asymptotic behavior of solutions of nonlinear partial differential equations as emphasized in the pioneering work \cite{diper1}. A key feature of these measures is their capacity to capture the oscillations of minimizing sequences of non convex variational problems see \cite{fonse2} among a wider bibliography.
The special properties of Young measures generated by sequences of gradients of Sobolev functions have been studied by Kinderlherer \& Pedegral \cite{kinder2,kinder1} and are relevant in the applications to nonlinear elasticity. Two-scale Young measures, which have been introduced in \cite{w} and \cite{pedregal3} to study periodic homogenization of nonlinear transport equations and integral functionals, contain some information on the amount of oscillations and extend Nguetseng's notion of two-scale convergence, thus entailing to get multiscale homogenization results for \eqref{intro1} in the classical Sobolev setting, e.g. the results in \cite{baba1} where two-scale ($W^{1,p}$)-gradient Young measures have been exaustively studied and applied.
 
\par Following this latter approach, in order to address the homogenization of \eqref{intro1}, with $f(x, \langle\frac{x}{\varepsilon}\rangle, \cdot)$ with non standard growth,  we consider Young measures generated by sequences of the type $\{(\langle\cdot/\varepsilon\rangle, \nabla u_{\varepsilon})\}$, which are, roughly speaking, what we will call \textit{two-scale $W^1L^{\Phi}$- gradient Young measures}. From a physical point of view, we seek to capture microstructures-due to finer and finer oscillations of minimizing sequences that cannot reach an optimal state- at a given scale $\varepsilon$ (period of the material heterogeneities). In this way, the minima of the limit problem captures two kinds of oscillations of the minimizing sequences:                                   those due to the periodic heterogeneities of the material and those due to a possible multi-well structure. Clearly the analysis of two-scale $W^1L^\Phi$-gradient Young measures contains information about   $W^1L^\phi$-gradient Young
measures. In fact this paper also contains explicit results regarding this latter framework. \color{black} 
\par In details, we extend  the results in \cite{kinder1} and \cite{baba1} to the Orlicz-Sobolev's spaces, giving a complete algebraic characterization of $W^{1}L^{\Phi}$- gradient Young measures and two-scale  $W^{1}L^{\Phi}$-gradient Young measures, in particular we obtain in the first case a  characterization in terms of a Jensen's inequality with test functions in the space $\mathscr{E}_\Phi$ of continuous functions $f : \mathbb{R}^{d\times N} \to \mathbb{R}$ such that the limit 
\begin{equation}\label{1.2}
	\lim_{|\xi|\to +\infty} \dfrac{f(\xi)}{1 + \Phi(|\xi|)},
\end{equation} 
exists, and in the space,  $\mathcal{E}_{\Phi}$ of continuous functions $f : \overline{Y}\times \mathbb{R}^{d\times N} \to \mathbb{R}$ such that the limit 
\begin{equation}\label{1.3}
	\lim_{|\xi|\to +\infty} \dfrac{f(y, \xi)}{1 + \Phi(|\xi|)},
\end{equation} 
exists uniformly with respect to $y\in \overline{Y}$.  
Namely, taking into account formulae \eqref{Dela2alpha} and \eqref{nabla2beta}, we obtain the following results.

\begin{thm}\label{c1theo0}
Let $\Omega$ be a bounded open subset of $\mathbb{R}^{N}$ with Lipschitz boundary,	let $\Phi$ be a Young function of class $\Delta_{2}\cap \nabla_2$, and let $\lambda \in L^{\infty}_{\omega}\left(\Omega; \mathcal{M}(\mathbb{R}^{d\times N})\right)$ be such that $\lambda_{x} \in \mathcal{P}(\mathbb{R}^{d\times N})$ for a.e. $x \in \Omega$. The family $\{\lambda_{x}\}_{x\in \Omega}$ is $W^{1}L^{\Phi}$- gradient Young measure if and only if the three conditions below hold : 
	\begin{itemize}
		\item[i)] there exists $u \in W^{1}L^{\Phi}(\Omega; \mathbb{R}^{d})$ 
		such that 
		\begin{equation}\label{c1eq10}
			\int_{\mathbb{R}^{d\times N}} \xi d\lambda_{x}(\xi) = \nabla u(x) \hbox{ for a.e.} \, x \in \Omega;  
		\end{equation}
		\item[ii)] for every $f \in \mathscr{E}_{\Phi}$, 
		\begin{equation}\label{c1eq20}
			\int_{\mathbb{R}^{d\times N}} f(\xi) d\lambda_{x}(\xi) \geq  {\mathcal Q} f(\nabla u(x)) \;\, \hbox{ for a.e.} \, x \in \Omega,  
		\end{equation}
	where $\mathcal Q$ f is the quasiconvexification of $f$ in the sense of Morrey, i.e.
	\begin{equation}\label{Qfdef}
		{\mathcal Q} f(\xi)=\inf\left\{\frac{1}{\mathcal L^N(\Omega)} \int_\Omega f(\xi + \nabla \varphi(y))dy: \varphi \in C^\infty_c(\Omega;\mathbb R^d)\right\}
	\end{equation}
	\color{black}
		\item[iii)] \begin{equation}\label{c1eq40}
			x \mapsto \int_{\mathbb{R}^{d\times N}} \Phi(|\xi|) d\lambda_{x}(\xi) \in L^{1}(\Omega).
		\end{equation}
	\end{itemize} 
\end{thm} 
\color{black}
\begin{thm}\label{c1theo1}
	Let $\Omega$ be a bounded open subset of $\mathbb{R}^{N}$ with Lipschitz boundary  and let $\Phi$ be a Young function of class $\Delta_{2}\cap \nabla_2$, 
	and let $\nu \in L^{\infty}_{\omega}\left(\Omega\times Y; \mathcal{M}(\mathbb{R}^{d\times N})\right)$ be such that $\nu_{(x,y)} \in \mathcal{P}(\mathbb{R}^{d\times N})$ for a.e. $(x,y) \in \Omega\times Y$. The family $\{\nu_{(x,y)}\}_{(x,y)\in \Omega\times Y}$ is a $W^{1}L^\Phi$-two-scale gradient Young measure if and only if the three conditions below hold : 
	\begin{itemize}
		\item[i)] there exist $u \in W^{1}L^{\Phi}(\Omega; \mathbb{R}^{d})$ and $u_{1} \in L^{1}(\Omega; W^{1}_{\#}L^{\Phi}_{per}(Y;\mathbb R^d))$ with $\left(u_1, \frac{\partial u_{1}}{\partial y_{i}}\right) \in L^{\Phi}(\Omega\times Y_{per};\mathbb R^{2d})$ ($1 \leq i \leq N$) such that 
		\begin{equation}\label{c1eq1}
		\int_{\mathbb{R}^{d\times N}} \xi d\nu_{(x,y)}(\xi) = \nabla u(x) + \nabla_{y}u_{1}(x,y) \;\, \textup{for \,a.e.} \, (x,y) \in \Omega\times Y;  
		\end{equation}
		\item[ii)] for every $f \in \mathcal{E}_{\Phi}$, 
		\begin{equation}\label{c1eq2}
		\int_{Y}\int_{\mathbb{R}^{d\times N}} f(x,\xi) d\nu_{(x,y)}(\xi) dy \geq  f_{hom}(\nabla u(x)) \;\, \textup{for \,a.e.} \, x \in \Omega,  
		\end{equation}
		where 
		\begin{equation}\label{c1eq3}
		f_{hom}(\xi) = \lim_{T\to +\infty} \inf_{\phi} \left\{ \dashint_{(0,T)^{N}} f(\langle y\rangle, \xi+\nabla\phi(y))dy\, :\, \phi \in  W^{1}_{0}L^{\Phi}((0,T)^{N}; \mathbb{R}^{d}) \right\};
		\end{equation}
		\item[iii)] \begin{equation}\label{c1eq4}
		(x, y) \mapsto \int_{\mathbb{R}^{d\times N}} \Phi(|\xi|) d\nu_{(x,y)}(\xi) \in L^{1}(\Omega\times Y).
		\end{equation}
	\end{itemize} 
\end{thm}
We note that both $\mathscr E_\Phi$ and $\mathcal{E}_{\Phi}$ are separable (see section \ref{sect3}), and thus it is sufficient to check conditions $(ii)$ in both theorems above for countably many test functions. 

The proof of both results are similar to the one first proposed in \cite{kinder1} for gradient Young measures and later extended to two-scale gradient Young measures in $W^{1,p}$ in \cite{baba1}. 

Due to the similarity of the statements, we prefer to present just the proof of the second result, also avoiding to reproduce the parts which are identical to the proof of \cite[Theorem 1.1]{baba1}.
On the other hand we emphasize that some of the implications of Theorem \ref{c1theo0} were contained in \cite[Theorem 1.2]{b}.

As regards as the proof of Theorem \ref{c1theo1}, as usual first the homogeneous case is addressed, i.e. considering two-scale gradient $W^{1}L^{\Phi}$- Young measures, independent of the macroscopic variable $x \in \Omega$, are considered. For these ones the result is achieved by means of the Hahn-Banach Seperation Theorem. Finally the general case is obtained by approximating general two-scale $W^{1}L^{\Phi}$- gradient Young measures by piecewise constant ones, with respect to the variable $x \in \Omega$. 
\par Theorem \ref{c1theo1} is a key tool to prove a representation theorem for the $\Gamma$-limit of \eqref{intro1} in terms of two-scale $W^{1}L^{\Phi}$- gradient Young measures, under  
very mild regularity hypotheses on the integrand $f$.
\begin{thm}\label{c1theo2}
	Let $\Omega$ be a bounded open subset of $\mathbb{R}^{N}$ with Lipschitz boundary, let $\Phi$ be a Young function of class $\Delta_{2}\cap \nabla_2$, 
	and let $f : \Omega\times Y\times \mathbb{R}^{d\times N} \to [0,+\infty)$ be an admissible integrand. Assume that there exist constants $\alpha, \beta$ and such that for all $(x,y, \xi) \in \Omega\times Y \times \mathbb{R}^{d\times N}$
	\begin{equation}\label{c1intro2}
	\alpha \Phi(|\xi|) \leq f(x,y, \xi) \leq \beta (1+\Phi(|\xi|)).
	\end{equation}
	Then the functional $\mathcal{F}_{\varepsilon}$ $\Gamma$-converges with respect to the weak $W^{1}L^{\Phi}(\Omega; \mathbb{R}^{d})$-topology (or equivalently the strong $L^{\Phi}(\Omega; \mathbb{R}^{d})$-topology) to $\mathcal{F}_{hom} : $ \\ $ W^{1}L^{\Phi}(\Omega; \mathbb{R}^{d})  \to [0,+\infty)$ given by 
	\begin{equation*}
	\mathcal{F}_{hom}(u) = \min_{\nu \in \mathcal{M}_{u}} \int_{\Omega}\int_{Y}\int_{\mathbb{R}^{d\times N}} f(x,y,\xi) d\nu_{(x,y)}(\xi) dy dx,
	\end{equation*}
	where 
	\begin{equation}\label{c1mu}
	\begin{array}{rcl}
	\mathcal{M}_{u} & := & \bigg\{ \nu \in  L^{\infty}_{\omega}\left(\Omega\times Y; \mathcal{M}(\mathbb{R}^{d\times N})\right) :\, \{\nu_{(x,y)}\}_{(x,y)\in\Omega\times Y}  \\
	  &   & \;\;\; \hbox{ is a two-scale}-W^1L^{\Phi}-\hbox{ gradient Young measure such that }  \\
	 & & \;\;\;\; \;\;\;\;\;\;\;\; \displaystyle{\nabla u(x) = \int_{\Omega} \int_{\mathbb{R}^{d\times N}} \xi d\nu_{(x,y)}(\xi) dy \;\, \hbox{ for  a.e.} \, x \in \Omega \bigg\}}
	\end{array}
	\end{equation}
	for all $u \in W^{1}L^{\Phi}(\Omega; \mathbb{R}^{d})$.
\end{thm} 
\par The paper is organized as follows: in section \ref{sect2} we present some preliminary results on $W^{1}L^{\Phi}$-gradient Young measures and  two-scale $W^{1}L^{\Phi}$-gradient Young measures.  Section \ref{sect3} is devoted to Theorems \ref{c1theo0} and \ref{c1theo1} while  the homogenization result Theorem \ref{c1theo2} is obtained in section \ref{sect4}.


\section{Preliminary results}\label{sect2}

\subsection{Notations}
In this section, we start fixing some notation, which for the readers' convenience is not very different from the one adopted in \cite{baba1}. We will also recall some preliminaries about Orlicz and Orlicz-Sobolev spaces that we will use in the sequel.
\begin{itemize}
	\item $\Omega$ is an open bounded subset of $\mathbb{R}^{N}$ with Lipschitz boundary.
	\item $\mathcal{A}(\Omega)$ denotes the family of all open subsets of $\Omega$.
	\item  $\mathcal{L}^{N}$ is the Lebesgue measure in $\mathbb{R}^{N}$.
	\item  $\mathbb{R}^{d\times N}$ is identified with the set of real $d\times N$ matrices.
	\item $Y := (0,1)^{N}$ is the unit cube in $\mathbb{R}^{N}$. 
	\item  The symbols $\langle\cdot\rangle$ and $[\cdot]$ stand, respectively, for the fractional and integer part of a number, or a vector, componentwise.
	\item  The Dirac mass at a point $a\in \mathbb{R}^{m}$ is denoted by $\delta_{a}$.
	\item The symbol $\dashint_{A}$ stands for the average $\mathcal{L}^{N}(A)^{-1}\int_{A}$.
	\item  $U$ is an open subset of $\mathbb{R}^{m}$.
	\item  $\mathcal{C}_{c}(U)$  is the space of continuous  functions $f: U\to \mathbb{R}$ with compact support.
	\item  $\mathcal{C}_{0}(U)$ is the closure of $\mathcal{C}_{c}(U)$ for the uniform convergence, it coincides with the space of all continuous functions $f: U\to \mathbb{R}$ such that, for every $\eta>0$, there exists a compact set $K_{\eta}\subset U$ with $|f|< \eta$ on $U\backslash K_{\eta}$.
	\item  $\mathcal{M}(U)$ is the space of real-valued Randon measures with finite total variation. We recall that by the Riesz Representation Theorem $\mathcal{M}(U)$ can be identified with the dual space of $\mathcal{C}_{0}(U)$ through the duality 
	\begin{equation*}
	\langle \mu, \phi\rangle = \int_{U} \phi d\mu, \quad \mu \in \mathcal{M}(U), \;\; \phi \in \mathcal{C}_{0}(U).
	\end{equation*}
	\item  $\mathcal{P}(U)$ denotes the space of probability measures on $U$, i.e. the space of all $\mu\in \mathcal{M}(U)$ such that $\mu \geqslant 0$ and $\mu(U)=1$.
	\item $L^{1}(\Omega;\mathcal{C}_{0}(U))$ is the space of maps $\phi: \Omega\to \mathcal{C}_{0}(U)$ such that 
	\begin{itemize}
		\item[i)] $\phi$ is strongly measurable, i.e. there exists a sequence of simple functions $s_{n}: \Omega\to \mathcal{C}_{0}(U)$ such $\|s_{n}(x)-\phi(x)\|_{\mathcal{C}_{0}(U)} \to 0$ for a.e. $x\in\Omega$,
		\item[ii)] $x\to \|\phi(x)\|_{\mathcal{C}_{0}(U)}\in L^{1}(\Omega)$.
	\end{itemize}
We recall that the linear space spanned by $\{\varphi\otimes\psi:\,\varphi\in L^{1}(\Omega)\, \textup{and}\, \psi\in\mathcal{C}_{0}(U)\}$ is dense in $L^{1}(\Omega;\mathcal{C}_{0}(U))$.
   \item $L^{\infty}_{\omega}(\Omega;\mathcal{M}(U))$ is the space of maps $\nu: \Omega \to \mathcal{M}(U)$ such that 
\begin{itemize}
	\item[i)] $\nu$ is weak$^{\ast}$ measurable, i.e. $x \to \langle\nu_{x},\phi\rangle$ is measurable for every $\phi\in\mathcal{C}_{0}(U)$,
	\item[ii)] $x\to \|\nu(x)\|_{\mathcal{M}(U)}\in L^{\infty}(\Omega)$.
\end{itemize}
The space $L^{\infty}_{\omega}(\Omega;\mathcal{M}(U))$ can be identified with the dual of $L^{1}(\Omega;\mathcal{C}_{0}(U))$ through the duality 
\begin{equation*}
\langle\mu,\phi\rangle = \int_{\Omega}\int_{U} \phi(x,\xi) d\mu_{x}(\xi) dx, \quad \mu\in L^{\infty}_{\omega}(\Omega;\mathcal{M}(U)), \;\; \phi\in L^{1}(\Omega;\mathcal{C}_{0}(U)), 
\end{equation*}
where $\phi(x,\xi):= \phi(x)(\xi)$ for all $(x,\xi) \in \Omega\times U$. 
\item $\Phi: [0,\infty)\to[0,\infty)$ is a Young function, i.e. $\Phi$ is continuous, convex, with $\Phi(t)>0$ for $t>0$, $\frac{\Phi(t)}{t}\to 0$ as $t\to 0$, and $\frac{\Phi(t)}{t}\to \infty$ as $t\to \infty$.
\item $\widetilde{\Phi}$ stands for the complementary of Young function $\Phi$, defined by 
\begin{equation*}
\widetilde{\Phi}(t)= \sup_{s\geq 0} \big\{st - \Phi(s), \; t\geq 0 \big\}.
\end{equation*} 
\item We recall that a Young function $\Phi$ is of class $\Delta_{2}$ at $\infty$ (denoted $\Phi\in \Delta_{2}$) if there are $\alpha>0$ and $t_{0}\geq 0$ such that 
\begin{equation}\label{Dela2alpha}\Phi(2t) \leq \alpha \Phi(t),\; \textup{for\,all}\, t\geq t_{0}.
	\end{equation}
 Also $\Phi$ is of class $\nabla_2$ if $\tilde \Phi$ is of class $\Delta_2$, i.e. $\exists \beta >0$ and $t_0 >1$ such that  
 \begin{equation}\label{nabla2beta}\Phi(t) \leq \frac{1}{2 \beta} \Phi(\beta t), \hbox{ for all }t \geq t_0.
 	\end{equation}
 
\item $L^{\Phi}(\Omega;\mathbb{R}^{d})$ is the Orlicz space of functions defined by 
\begin{equation*}
L^{\Phi}(\Omega;\mathbb{R}^{d}) = \left\{ u: \Omega\to \mathbb{R}^{d}\,;\, u\,\textup{is\,measurable},\; \lim_{\alpha\to 0}\int_{\Omega}\Phi(\alpha|u(x)|)dx = 0 \right\}.
\end{equation*}
We recall that $L^{\Phi}(\Omega;\mathbb{R}^{d})$ is a Banach space with respect to the Luxemburg norm 
\begin{equation*}
\|u\|_{\Phi} = \inf \left\{ k>0\,:\, \int_{\Omega}\Phi\left(\dfrac{|u(x)|}{k}\right)dx \leq 1 \right\}\, < \, +\infty.
\end{equation*}
Sometimes, we will denote the norm of elements in $L^{\Phi}(\Omega;\mathbb{R}^{d})$, both by  $\|\cdot\|_{\Phi}$ and with $\|\cdot\|_{L^{\Phi}}$.
\item   $\mathcal{D}(\Omega)$ is the space of  indefinitely differentiable functions $f: \Omega\to \mathbb{R}^{d}$ with compact support. We recall that $\mathcal{D}(\Omega)$ is dense in $L^{\Phi}(\Omega;\mathbb{R}^{d})$, $L^{\Phi}(\Omega;\mathbb{R}^{d})$ is separable and reflexive when $\Phi, \in \Delta_{2}\cap \nabla_2$.  The dual of $L^{\Phi}(\Omega;\mathbb{R}^{d})$ is identified with $L^{\widetilde{\Phi}}(\Omega;\mathbb{R}^{d})$. The property $\frac{\Phi(t)}{t}\to \infty$ as $t\to \infty$ implies that 
\begin{equation*}
L^{\Phi}(\Omega;\mathbb{R}^{d}) \subset L^{1}(\Omega;\mathbb{R}^{d}) \subset L^{1}_{loc}(\Omega;\mathbb{R}^{d}) \subset \mathcal{D}'(\Omega),
\end{equation*}
each embedding being continuous.
\item $W^{1}L^{\Phi}(\Omega,\mathbb{R}^{d})$ is the Orlicz-Sobolev space defined by
\begin{equation*}
W^{1}L^{\Phi}(\Omega,\mathbb{R}^{d}) = \left\{ u \in L^{\Phi}(\Omega;\mathbb{R}^{d})\,;\, \dfrac{\partial u}{\partial x_{i}} \in L^{\Phi}(\Omega;\mathbb{R}^{d}),\, 1\leq i\leq N \right\},
\end{equation*}
where derivatives are taken in the distributional sense on $\Omega$. Endowed with the norm 
\begin{equation*}
\|u\|_{W^{1}L^{\Phi}} = \|u\|_{L^{\Phi}} + \sum_{i=1}^{N} \left\|\dfrac{\partial u}{\partial x_{i}} \right\|_{L^{\Phi}}, \; u \in W^{1}L^{\Phi}(\Omega,\mathbb{R}^{d}),
\end{equation*}
$W^{1}L^{\Phi}(\Omega,\mathbb{R}^{d})$ is a reflexive Banach space when $\Phi \in \Delta_{2}\cap \nabla_2$. 
\item $W^{1}_{0}L^{\Phi}(\Omega,\mathbb{R}^{d})$ denotes the closure of $\mathcal{D}(\Omega)$ in $W^{1}L^{\Phi}(\Omega,\mathbb{R}^{d})$ and the semi-norm 
\begin{equation*}
u \longrightarrow \|u\|_{W^{1}_{0}L^{\Phi}} = \|Du\|_{L^{\Phi}} =  \sum_{i=1}^{N} \left\|\dfrac{\partial u}{\partial x_{i}} \right\|_{L^{\Phi}}
\end{equation*}
is a norm (of gradient) on $W^{1}_{0}L^{\Phi}(\Omega,\mathbb{R}^{d})$ equivalent to $\|\cdot\|_{W^{1}L^{\Phi}}$.
\item Given a function space $\mathcal{S}$ defined in $\Omega$, $Y$ or $\Omega\times Y$, the subscript $\mathcal{S}_{per}$ means that the functions are periodic in $\Omega$, $Y$ or $\Omega\times Y$, as it will be clear from the context.
	\item
$L^\Phi(\Omega \times \mathbb R^N_{loc})$ is the space defined by
\begin{align*} L^\Phi(\Omega \times \mathbb R^N_{loc}):=\left\{v \hbox{ measurable, such that } \right.\\ \left. \iint_{\Omega \times A} \Phi(v(x,y))dxdy<+\infty \hbox{ for every }A \subset \subset \mathbb R^N_y\right\}.
	\end{align*}
\item $L^{\Phi}(\Omega\times Y_{per})$ is the Orlicz space defined by
\begin{equation*}
L^{\Phi}(\Omega\times Y_{per}) = \left\{ v \in L^{\Phi}(\Omega\times \mathbb{R}^{N}_{loc}): v(x,\cdot)\, \hbox{ is }\, Y-\textup{periodic} \hbox{ for a.e. } x \in \Omega \right\}.
\end{equation*}
\item $W^{1}L^{\Phi}_{per}(Y)$ is the Orlicz-Sobolev space defined by
\begin{equation*}
	\begin{array}{ll}
W^{1}L^{\Phi}_{per}(Y) = \left\{ v \in W^{1}L^{\Phi}_{loc}(\mathbb{R}^{N}_{y}): v(x,\cdot)\,\textup{and}\,\dfrac{\partial v}{\partial y_{i}}(x,\cdot),\, 1\leq i\leq N, \right.\\
\left. \hbox{ are }\,\ Y-\textup{periodic} \hbox{ for a.e. }x \in \Omega\right\}.
\end{array}
\end{equation*}
\item $W^{1}_{\#}L^{\Phi}(Y)$ is the Orlicz-Sobolev space defined by
\begin{equation*}
W^{1}_{\#}L^{\Phi}(Y) = \left\{ v \in W^{1}L^{\Phi}_{per}(Y)\,:\, \dashint_{Y}v(y) dy = 0 \right\}.
\end{equation*}
It is endowed with the $L^\Phi$ norm of the gradients.
\end{itemize}

\subsection{(Two-scale) $W^{1}L^{\Phi}$-gradient Young measures}

In this subsection we assume that $\Phi$ is of class $\Delta_2 \cap \nabla_2$, which, as proven in \cite{DG}, is equivalent at requiring that both $\Phi$ and its conjugate are of class $\Delta_2$. It is also worth to observe that in this subsection the value of the constants $\alpha$ and $\beta$ in \eqref{Dela2alpha} and $\eqref{nabla2beta}$ is not crucial. It will only play a role in the characterization of our main theorems \ref{c1theo1} and \ref{c1theo2}.

We start first by recalling the notion of Young measure and some of its properties. We refer to \cite{baba1, FLbook, pedregal2} for a detailed description on the subject. 
\begin{defn}(Young measure)\\
	Let $\nu \in L^{\infty}_{\omega}(\Omega;\mathcal{M}(\mathbb{R}^{m}))$ and let $z_{n}: \Omega\to\mathbb{R}^{m}$ be a sequence of measurable functions. The family of measures $\{\nu_{x}\}_{x\in\Omega}$ is said to be the Young measure generated by $\{z_{n}\}$ provided $\nu_{x}\in\mathcal{P}(\mathbb{R}^{m})$ for a.e. $x\in \Omega$ and 
	\begin{equation*}
	\delta_{z_{n}} \;\, \overset{\ast}{\rightharpoonup}  \;\; \nu \;\, in \;\, L^{\infty}_{\omega}(\Omega;\mathcal{M}(\mathbb{R}^{m})),
	\end{equation*}
	i.e. for all $\psi \in L^{1}(\Omega;\mathcal{C}_{0}(\mathbb{R}^{m}))$
	\begin{equation*}
	\lim_{n\to +\infty} \int_{\Omega} \psi(x,z_{n}(x)) dx = \int_{\Omega}\int_{\mathbb{R}^{m}} \psi(x,\xi) d\nu_{x}(\xi) dx.
	\end{equation*}
\end{defn}

The family $\{\nu_{x}\}_{x\in\Omega}$ above defined is said to be  \textit{a homogeneous Young measure} if the map $x\to \nu_{x}$ is independent of $x$, hence $\{\nu_{x}\}_{x\in\Omega}$ can be identified with a single element $\nu$ of $\mathcal{M}(\mathbb{R}^{m})$. 

The existence of Young measures is well known and it is asserted in the so called Fundamental Theorem (see e.g. \cite{ball1},\cite[Theorem 2.2]{FMAq}), that we write for the readers' convenience.

\begin{prop}\label{ball1}
	Let $\{z_{n}\}$ be a sequence of measurable functions $z_{n}: \Omega\to\mathbb{R}^{m}$. Then there exist a subsequence $\{z_{n_{k}}\}$ and $\nu \in L^{\infty}_{\omega}(\Omega;\mathcal{M}(\mathbb{R}^{m}))$ with $\nu_{x}\geqslant 0$ for a.e. $x\in \Omega$, such that $\delta_{z_{n_{k}}}\, \overset{\ast}{\rightharpoonup} \, \nu$ in $L^{\infty}_{\omega}(\Omega;\mathcal{M}(\mathbb{R}^{m}))$ and the following properties hold:	
	\begin{itemize}
	\item[i)]  $\|\nu_{x}\|_{\mathcal{M}(\mathbb{R}^{m})} = \nu_{x}(\mathbb{R}^{m}) \leqslant 1$ for a.e. $x\in\Omega$;
		\item[ii)] if $\textup{dist}(z_{n_{k}}, K) \to 0$ in measure for some closed set $K\subset \mathbb{R}^{m}$, then $\textup{Supp}(\nu_{x}) \subset K$ for a.e. $x\in \Omega$;
				\item[iii)]  $\|\nu_{x}\|_{\mathcal{M}(\mathbb{R}^{m})} = 1$ if and only if there exists a Borel function $g: \mathbb{R}^{m}\to [0,+\infty]$ such that 
		\begin{equation*}
		\lim_{|\xi|\to +\infty} g(\xi) = +\infty \;\; \textup{and} \;\; \sup_{k\in \mathbb{N}} \int_{\Omega} g(z_{n_{k}}(x)) dx < +\infty;
		\end{equation*}
		\item[iv)] if $f: \Omega\times\mathbb{R}^{m}\to [0,+\infty]$ is a normal integrand, then 
		\begin{equation*}
		\liminf_{k\to +\infty} \int_{\Omega} f(x, z_{n_{k}}(x)) dx \geqslant \int_{\Omega}\int_{\mathbb{R}^{m}} f(x,\xi) d\nu_{x}(\xi) dx;
		\end{equation*}
		\item[v)] if (iii) holds and if $f: \Omega\times\mathbb{R}^{m}\to [0,+\infty]$ is a Carath\'{e}odory integrand such that the sequence $\{f(\cdot, z_{n_{k}})\}$ is equi-integrable then 
		\begin{equation*}
		\liminf_{k\to +\infty} \int_{\Omega} f(x, z_{n_{k}}(x)) dx = \int_{\Omega}\int_{\mathbb{R}^{m}} f(x,\xi) d\nu_{x}(\xi) dx.
		\end{equation*} 
	\end{itemize}
\end{prop}

\color{black}

We are interested in extending the results of \cite{kinder1, kinder2, baba1} to the Orlicz setting, describing the Young measure generated by $\{\nabla u_{\varepsilon}\}$ and by the pair $\{(\langle\frac{\cdot}{\varepsilon}\rangle, \nabla u_{\varepsilon})\}$, respectively, for some sequence $\{u_{\varepsilon}\} \subset W^{1}L^{\Phi}(\Omega;\mathbb{R}^{d})$.

Concerning the second case, we recall that, if  $\mu \in L^{\infty}_{\omega}(\Omega;\mathcal{M}(\mathbb{R}^{N}\times\mathbb{R}^{d\times N}))$ and $\{u_{\varepsilon}\} \subset W^{1}L^{\Phi}(\Omega;\mathbb{R}^{d})$ is such that $\{(\langle\frac{\cdot}{\varepsilon}\rangle, \nabla u_{\varepsilon})\}$ generates the Young measure $\{\mu_{x}\}_{x\in\Omega}$, since  Riemann-Lebesgue Lemma (see e.g. \cite[Example 3]{LNW}) entails that $\{\langle\frac{\cdot}{\varepsilon}\rangle\}$ generates the homogeneous Young measure $dy := \mathcal{L}^{N}\lfloor Y$ (restriction of the Lebesgue measure to $Y$),  then by the Disintegration Theorem of a measure on a product (see \cite{vala1}), there exists a map $\nu \in L^{\infty}_{\omega}(\Omega\times Y;\mathcal{M}(\mathbb{R}^{d\times N}))$ with $\nu_{(x,y)} \in \mathcal{P}(\mathbb{R}^{d\times N})$ for a.e. $(x,y) \in \Omega\times Y$ and such that $\mu_{x} = \nu_{(x,y)}\otimes dy$ for a.e. $x\in\Omega$, i.e. 
\begin{equation*}
\int_{\mathbb{R}^{N}\times\mathbb{R}^{d\times N}} \phi(y,\xi) d\mu_{x}(y,\xi) = \int_{Y}\int_{\mathbb{R}^{d\times N}} \phi(y,\xi) d\nu_{(x,y)}(\xi) dy
\end{equation*}
for every $\phi\in \mathcal{C}_{0}(\mathbb{R}^{N}\times\mathbb{R}^{d\times N})$.
\par Following \cite{pedregal3} and \cite{baba1}, the family $\{\nu_{(x,y)}\}_{(x,y)\in\Omega\times Y}$ is called \textit{two-scale $W^{1}L^{\phi}$-(gradient) Young measure} associated to  $\{u_n\}$ at scale $\varepsilon$, while the Young measure generated by $\{\nabla u_\varepsilon\}\subset W^{1}L^{\Phi}(\Omega;\mathbb R^d)$ will be simply called $W^{1}L^{\Phi}$-gradient Young measure, to emphasize that they are extensions to the Orlicz-Sobolev setting of the notions of gradient Young measures introduced in \cite{kinder1, kinder2}, while the two-scale $W^{1}L^{\Phi}$-gradient Young measures can be seen as an extension of \cite[Definition 2.4]{pedregal3} to sequences of gradients of fields in $W^{1}L^{\phi}$. More precisely we have the following definition.

\color{black}
\begin{defn}\label{c2eq1} \textit{(two-scale) $W^{1}L^{\Phi}$- gradient Young measure} \\ 
\begin{itemize}
	\item[i)] Let $\lambda \in L^{\infty}_{\omega}(\Omega;\mathcal{M}(\mathbb{R}^{d\times N}))$. The family $\{\lambda_{x}\}_{x\in\Omega}$ is said to be a $W^{1}L^{\Phi}$-gradient Young measure if $\lambda_{x} \in \mathcal{P}(\mathbb{R}^{d\times N})$ for a.e. $x \in \Omega$ and if for every sequence $\{\varepsilon_{n}\} \to 0$ there exist a bounded sequence $\{u_{n}\}$ in $W^{1}L^{\Phi}(\Omega;\mathbb{R}^{d})$ such that $\{\nabla u_{n}\}$ generates the Young measure $\{\lambda_{x}\}_{x\in\Omega}$ i.e. for every $z\in L^{1}(\Omega)$ and $\varphi\in\mathcal{C}_{0}(\mathbb{R}^{d\times N})$, 
	\begin{equation*}
		\lim_{n\to +\infty} \int_{\Omega} z(x) \varphi\left(\nabla u_{n}(x)\right) dx = \int_\Omega\int_{\mathbb{R}^{d\times N}} z(x) \varphi(\xi) d\lambda_{x}(\xi) dx. 
	\end{equation*}
$\{\lambda_{x}\}_{x\in\Omega}$ is also called the $W^{1}L^{\Phi}$- Young gradient measure associated to $\{\nabla u_n\}$.
	\item[ii)]	Let $\nu \in L^{\infty}_{\omega}(\Omega\times Y;\mathcal{M}(\mathbb{R}^{d\times N}))$. The family $\{\nu_{(x,y)}\}_{(x,y)\in\Omega\times Y}$ is said to be a two-scale $W^{1}L^{\Phi}$-gradient Young measure if $\nu_{(x,y)} \in \mathcal{P}(\mathbb{R}^{d\times N})$ for a.e. $(x,y) \in \Omega\times Y$ and if for every sequence $\{\varepsilon_{n}\} \to 0$ there exist a bounded sequence $\{u_{n}\}$ in $W^{1}L^{\Phi}(\Omega;\mathbb{R}^{d})$ such that $\{(\langle\frac{\cdot}{\varepsilon_{n}}\rangle, \nabla u_{n})\}$ generates the Young measure $\{\nu_{(x,y)}\otimes dy\}_{x\in\Omega}$ i.e. for every $z\in L^{1}(\Omega)$ and $\varphi\in\mathcal{C}_{0}(\mathbb{R}^{N}\times\mathbb{R}^{d\times N})$, 
	\begin{equation*}
	\lim_{n\to +\infty} \int_{\Omega} z(x) \varphi\left(\langle\dfrac{x}{\varepsilon_{n}}\rangle, \nabla u_{n}(x)\right) dx = \int_{\Omega}\int_{Y}\int_{\mathbb{R}^{d\times N}} z(x) \varphi(y,\xi) d\nu_{(x,y)}(\xi) dy dx. 
	\end{equation*}
	In this case $\{\nu_{(x,y)}\}_{(x,y)\in\Omega\times Y}$ is also called the two-scale $W^{1}L^{\Phi}$- gradient Young measure associated to $\{\nabla u_n\}$.
	\end{itemize} 
\end{defn}
\begin{rem}\label{linkrem}
Putting together (i) and (ii) above, we have that
\begin{equation}\label{tsYMaveYM}
	\int_{\Omega}\int_{Y}\int_{\mathbb{R}^{d\times N}} z(x) \varphi(\xi) d\nu_{(x,y)}(\xi) dy dx=\int_\Omega\int_{\mathbb{R}^{d\times N}} z(x) \varphi(\xi) d\lambda_{x}(\xi) dx,
	\end{equation}
for every $\varphi \in C_0(\mathbb R^{d \times N})$, since on the left hand side we are integrating on $\Omega$, the function $\varphi$ can be considered as constant with respect the $y$ variable on $\Omega$,   and $z \in L^1(\Omega)$, 
thus, by the Fundamental Theorem on Calculus of Variations, one can conclude that $\lambda_x = \nu_{x,y}\otimes \frac{1}{\mathcal L^N(Y)}dy$, for a.e. $x \in \Omega$.
\end{rem}
\begin{ex}
	Let $\{\varepsilon_{n}\}$ be a sequence of positive numbers converging $0$. We will show some examples of gradient $W^{1}L^\Phi$  and two-scale gradient $ W^{1}L^\Phi$- Young measures.
	
First, if $N=2$ and $d=1$ we can consider the function $u_{\varepsilon_n}(x_1,x_2)= {\varepsilon_n}\varrho\left(\frac{x_1}{\varepsilon}\right)x_2$ with $\varrho$  defined in $[0,2]$ as follows 
$\varrho(x)=\left\{\begin{array}{ll}
x &\hbox{ if } x \in [0,1],\\
2-x &\hbox{ if }x \in [1,2]\end{array}\right.$
and then  periodically extended to $\mathbb R$. 

Then the sequence $\{\nabla u_{\varepsilon_n}\}$ generates the measure $(\frac{1}{2}\delta_{-x_2}+ \frac{1}{2}\delta_{x_2}, 0)$.

\medskip
The same examples provided in \cite[Examples 2.4 and 2.5]{baba1} guarantee existence of gradient Young measures generated by sequences in Orlicz-Sobolev spaces. We sketch them for the readers' convenience.

	Consider  $u : \Omega\to \mathbb{R}^{d}$ and $u_{1} : \Omega\times\mathbb{R}^{N}\to\mathbb{R}^{d}$ smooth functions such that $u_{1}(x,\cdot)$ is $Y$-periodic for all $x\in\Omega$. Define 
	\begin{equation*}
	u_{n}(x) := u(x) + \varepsilon_{n} u_{1}\left(x,\dfrac{x}{\varepsilon_{n}}\right).
	\end{equation*}
	Then, the two-scale $W^{1}L^{\Phi}$- gradient Young measure $\{\nu_{(x,y)}\}_{(x,y)\in\Omega\times Y}$ associated to $\{\nabla u_{n}\}$ is given by 
	\begin{equation*}
	\nu_{(x,y)} := \delta_{\nabla u(x)+\nabla_{y}u_{1}(x,y)} \quad \textup{for\,all}\, (x,y)\in\Omega\times Y.
	\end{equation*}
	Indeed, we follow line by line the arguments in \cite[Example 2.4]{baba1}, replacing \cite[Example 3]{LNW} by \cite[Proposition 4.3]{tacha1}.
	
	In particular, if $\varphi \in C_0(\mathbb R^{d\times N})$ and $\psi$ is a smooth function with compact support containing $\Omega$, such that $\theta =1$ in $\Omega$, then $\theta \varphi \in C_0(\mathbb R^d \times \mathbb R^{d \times N})$ and  
 the	$L^\Phi$- Young measure generated by $\{\nabla u_n\}$  can be computed as follows
		\begin{align*}
				&\displaystyle{\lim_{n\to +\infty} \int_{\Omega} \psi(x) \theta\left(\langle\dfrac{x}{\varepsilon_{n}}\rangle\right)\varphi\left( \nabla u_{n}(x) + \nabla_{y}u_{1}\left(x,\dfrac{x}{\varepsilon_{n}}\right) \right) dx }\\
			  = &\displaystyle{\int_{\Omega}\int_{Y} \psi(x) \varphi(\nabla u(x) + \nabla_{y}u_{1}(x,y)) dy dx} 
			\end{align*}
and it coincides with $\int_Y\delta_{\nabla u(x)+\nabla_y u_1(x,y)} dy$.

	Consider $\{\varepsilon_{n}\} \to 0$,  $u : \Omega\to \mathbb{R}^{d}$ and $u_{2} : \Omega\times\mathbb{R}^{N}\times \mathbb R^N\to\mathbb{R}^{d}$ smooth functions such that $u_{2}(x,\cdot,\cdot)$ is separately $Y$-periodic for all $x\in\Omega$. Define 
\begin{equation*}
	u_{n}(x) := u(x) + \varepsilon^2_{n} u_{2}\left(x,\frac{x}{\varepsilon_{n}},\frac{x}{\varepsilon^2_{n}}\right).
\end{equation*}
Moreover, arguing as in \cite[Example 2.5]{baba1} and exploiting the arguments used in \cite[Corollary 4.3]{FTNZDIE}, we obtain that 
the (two-scale )$L^{\Phi}$-Young measure associated to $\{\nabla u_n\}$ is
$\nu(x,y) =\int_Y \delta_{\nabla u(x)+ \nabla_z u_2(x,y,z)}dz$ for a.e. $(x,y) \in \Omega \times Y$,
while the $W^1L^{\Phi}$-gradient Young measure is
$\lambda_x =\iint_{Y\times Y} \delta_{\nabla u(x)+ \nabla_z u_2(x,y,z)}dzdy$ for a.e. $x \in \Omega$.
\end{ex}

\begin{rem}\label{c2rem2}
Arguing as in \cite[Remark 2.6]{baba1} it results that for every
	 $\{\varepsilon_{n} \}$, $\{u_{n} \}$ and $\nu$ as in Definition \ref{c2eq1}, without loss of generality  we can say that there exists $u\in W^{1}L^{\Phi}(\Omega;\mathbb{R}^{d})$ (with zero average), called \textit{underlying deformation} of $\{\nu_{(x,y)}\}_{(x,y)\in\Omega\times Y}$, such that 
	\begin{equation*}
	\nabla u(x) = \int_{Y} \int_{\mathbb{R}^{d\times N}} \xi d\nu_{(x,y)}(\xi)dy\quad \hbox{ a.e. in } \Omega,
	\end{equation*}
	
	Clearly the above is the two-scale version of the equality
	$$
		\nabla u(x) = \int_{\mathbb{R}^{d\times N}} \xi d\lambda_{x}(\xi)\hbox{ a.e.  in }\Omega,
	$$
which holds	for $\{\lambda_x\}_{x \in \Omega}$ being the $W^{1}L^{\Phi}$-Young measure generated by the sequence $\{\nabla u_n\}$, weakly converging to $\nabla u$ in $W^{1}L^{\Phi}(\Omega;\mathbb R^d)$,
and 
$u_{n}\, \rightharpoonup\, u $ in $W^{1}L^{\Phi}(\Omega;\mathbb{R}^{d})$. 
\end{rem}

\par The following lemma, whose proof is quite standard, (see e.g. \cite[Lemma 2.7]{baba1}), proves that there is no loss of generality in assuming that sequences of generators in Definition \ref{c2eq1} match the boundary condition of the underlying deformation.
\begin{lem}\label{c2lem1}
	Let $\{\varepsilon_{n}\} \to 0$ and $\{u_{n}\}\subset W^{1}L^{\Phi}(\Omega;\mathbb{R}^{d})$ be such that $u_{n}\, \rightharpoonup\, u $ in $W^{1}L^{\Phi}(\Omega;\mathbb{R}^{d})$ for some $u \in W^{1}L^{\Phi}(\Omega;\mathbb{R}^{d})$. Suppose that $\{\nabla u_n\}$ and $\{(\langle\cdot/\varepsilon_{n}\rangle, \nabla u_{n})\}$ generates the $W^{1}L^{\Phi}$-Young measure $\{\lambda_x\}_{x \in \Omega}$ and $\{\nu_{(x,y)}\otimes dy\}_{x\in\Omega}$, respectively. Then there exists a sequence $\{v_{n}\}\subset W^{1}L^{\Phi}(\Omega;\mathbb{R}^{d})$ such that $v_{n}\, \rightharpoonup\, u $ in $W^{1}L^{\Phi}(\Omega;\mathbb{R}^{d})$, $v_{n}=u$ on a neighborhood of $\partial\Omega$ and $\{\nabla v_n\}$ and  $\{(\langle\cdot/\varepsilon_{n}\rangle, \nabla v_{n})\}$ also generates $\{\lambda_x\}_{x\in \Omega}$ and $\{\nu_{(x,y)}\otimes dy\}_{x\in\Omega}$, respectively.
\end{lem}

\color{black}
$W^{1}L^{\Phi}$-gradient  Young measures $\{\lambda_x\}_{x\in \Omega} \in L^\infty_\omega(\Omega;\mathcal M(\mathbb R^{d \times N}))$ and (two-scale) $W^{1}L^{\Phi}$-gradient Young measures $\{\nu_{(x,y)}\}_{(x,y)\in\Omega\times Y}$ are said to be homogeneous if the  $x \mapsto \lambda_x$ and $(x,y)\mapsto \nu_{(x,y)}$ are independent of $x$, hence $\lambda$ and $\nu$ can be identified with elements of $\mathcal M(\mathbb R^{d \times N})$ and of $L^{\infty}_{\omega}(Y;\mathcal{M}(\mathbb{R}^{d\times N}))$, respectively with $\lambda \equiv \{\lambda_x\}_{x \in \Omega}$ and $\{\nu_{y}\}_{y\in Y} \equiv \{\nu_{(x,y)}\}_{(x,y)\in\Omega\times Y}$, respectively. 

Next we define the average of maps $\nu \in L^{\infty}_{\omega}(\Omega\times Y;\mathcal{M}(\mathbb{R}^{d\times N}))$ for which 
 $\{\nu_{(x,y)}\}_{(x,y)\in\Omega\times Y}$ is a two-scale $W^{1}L^\Phi$ gradient-Young measure. To this end we recall \cite[Definition 2.8]{baba1}. 
 \begin{defn}\label{c2def3}
 	Let $\nu \in L^{\infty}_{\omega}(\Omega\times Y;\mathcal{M}(\mathbb{R}^{d\times N}))$ be such that
 	$\{\nu_{(x,y)}\}_{(x,y)\in\Omega\times Y}$ is a two-scale $W^1L^{\Phi}$-gradient Young measure. The average of $\{\nu_{(x,y)}\}_{(x,y)\in\Omega\times Y}$ (with respect to the variable $x$) is the family $\{\overline{\nu}_{y}\}_{y\in Y}$ defined by 
 	\begin{equation*}
 	\langle\overline{\nu}_{y},\varphi\rangle := \dashint_{\Omega} \int_{\mathbb{R}^{d\times N}} \varphi(\xi) d\nu_{(x,y)} dx
 	\end{equation*}
 	for every $\varphi\in\mathcal{C}_{0}(\mathbb{R}^{d\times N})$.
 \end{defn}
 
	In particular if $\{\nu_{(x,y)}\}_{(x,y)\in\Omega\times Y}$ is a two-scale gradient Young measure, then, as observed in \cite[Remark 2.9]{baba1} $\overline{\mu} := \overline{\nu}_{y}\otimes dy$ is the average of $\{\mu_{x}\}_{x\in\Omega}$ with $\mu_{x} := \nu_{(x,y)}\otimes dy$ and $\mu \in L^{\infty}_{\omega}(\Omega;\mathcal{M}(\mathbb{R}^{N}\times\mathbb{R}^{d\times N}))$. Thus, $\bar \mu$ is a homogeneous Young measure.

We recall the following  result, which will be used in the applications to homogenization problems. The proof is given in \cite[Lemma 2.9]{baba1} in $W^{1,p}$, but it can be reproduced word by word constructing approximating sequences belonging to Sobolev-Orlicz spaces. 
It shows that $\{\overline{\nu}_{y}\}_{y\in Y}$ is actually a homogeneous two-scale $W^1L^{\Phi}$-gradient Young measure in the case where the underlying deformation is affine.


\begin{prop}\label{c2lem3}
	Let $\nu \in L^{\infty}_{\omega}(Y\times Y;\mathcal{M}(\mathbb{R}^{d\times N}))$ be such that
	$\{\nu_{(x,y)}\}_{(x,y)\in Y\times Y}$ is a two-scale $W^1L^{\Phi}$-gradient Young measure with underlying deformation $F\cdot$, for $F\in\mathbb{R}^{d\times N}$. Then $\{\overline{\nu}_{y}\}_{y\in Y}$ is a homogeneous two-scale $W^1L^{\Phi}$-gradient Young measure with the same underlying deformation. 
\end{prop}

\color{black}
\section{Theorems \ref{c1theo0} and  \ref{c1theo1}}\label{sect3}

We aim at presenting the arguments to prove Theorems \ref{c1theo0} and \ref{c1theo1}. To this end, we present the class of `test functions' $f$ used for characterizing $W^{1}L^{\Phi}$-gradient Young measures and two-scale $W^{1}L^{\Phi}$-gradient Young measires. Given a Young function $\Phi$ of class $\Delta_{2}$, we recall the spaces $\mathscr{E}_{\Phi}$  and $\mathcal E_{\Phi}$ introduced in \eqref{1.2} and \eqref{1.3} 

	The spaces are clearly not empty since the $\Phi(\xi)$ and the function $(y,\xi) \mapsto a(y) \Phi(|\xi|)$, where $a\in\mathcal{C}(\overline{Y})$, are in $\mathscr{E}_{\Phi}$ and in $\mathcal{E}_{\Phi}$, respectively, in particular $\mathscr E_\Phi \subset \mathcal E_\Phi$.

Standard arguments prove that $\mathscr{E}_{\Phi}$ and $\mathcal{E}_{\Phi}$ are Banach spaces under the norm
\begin{equation*}
\|f\|_{\mathcal{E}_{\Phi}} := \sup_{y\in\overline{Y},\xi\in\mathbb{R}^{d\times N}} \dfrac{|f(y,\xi)|}{1 + \Phi(|\xi|)},
\end{equation*}

Clearly if $f \in \mathscr E_\Phi$ the above norm specializes as
\begin{equation*}
	\|f\|_{\mathscr{E}_{\Phi}} := \sup_{\xi\in\mathbb{R}^{d\times N}} \dfrac{|f(\xi)|}{1 + \Phi(|\xi|)}.
\end{equation*}

In addition, $\mathscr E_{\Phi}$ and  $\mathcal{E}_{\Phi}$ are isomorphic to  $\mathcal{C}(\mathbb{R}^{d\times N}\cup\{\infty\})$ and  $\mathcal{C}(\overline{Y}\times(\mathbb{R}^{d\times N}\cup\{\infty\}))$, respectively  under the map 
\begin{equation*}\begin{array}{rcl}
\mathcal{E}_{\Phi} & \rightarrow & \mathcal{C}(\overline{Y}\times(\mathbb{R}^{d\times N}\cup\{\infty\})) \\
f & : & (y,\xi) \mapsto \left\{\begin{array}{lcl}
\dfrac{f(y,\xi)}{1 + \Phi(|\xi|)} & \textup{if} & (y,\xi) \in \overline{Y}\times \mathbb{R}^{d\times N} \\
 \lim_{|\xi|\to +\infty} \dfrac{f(y,\xi)}{1 + \Phi(|\xi|)} & \textup{if} & |\xi|=+\infty,
\end{array}  \right.
\end{array}
\end{equation*}
where $\mathbb{R}^{d\times N}\cup\{\infty\}$ denotes the one-point compactification of $\mathbb{R}^{d\times N}$, which becomes 
\begin{equation*}\begin{array}{rcl}
\mathscr E_{\Phi}  & \rightarrow & \mathcal{C}(\mathbb{R}^{d\times N}\cup\{\infty\}) \\
g & : & \xi \mapsto \left\{\begin{array}{lcl}
	\dfrac{g(\xi)}{1 + \Phi(|\xi|)} & \textup{if} &\xi \in  \mathbb{R}^{d\times N} \\
	\lim_{|\xi|\to +\infty} \dfrac{g(\xi)}{1 + \Phi(|\xi|)} & \textup{if} & |\xi|=+\infty,\end{array}  \right.\\
\end{array}
\end{equation*}
when $f$ does not depend on $y$, i.e. $f(y,\xi)= g(\xi)$ for every $(y,\xi) \in Y \times \mathbb R^{d \times N}$.
Consequently, both the spaces are separable, observing, in particular that $\mathscr E_{\Phi}$ is a closed subset of $\mathcal E_{\Phi}$. \\
Furthermore, for all $g \in \mathscr E_{\Phi}$ and $f\in \mathcal{E}_{\Phi}$, there exists a constant $c>0$ such that 
\begin{align}\label{c2eq6}
\max\{ |g(\xi)|,|f(y,\xi)| \}\leqslant c(1+\Phi(|\xi|)), \;\; \textup{for\,all}\, (y,\xi) \in \overline{Y}\times \mathbb{R}^{d\times N}. 
\end{align}
We denote by $(\mathscr{E}_{\Phi})'$ $(\mathcal{E}_{\Phi})'$ the dual spaces of $\mathscr E_{\Phi}$ and $\mathcal{E}_{\Phi}$, respectively  and the brackets $\langle\cdot,\cdot\rangle_{(\mathscr{E}_{\Phi})',\mathscr{E}_{\Phi}}$ $\langle\cdot,\cdot\rangle_{(\mathcal{E}_{\Phi})',\mathcal{E}_{\Phi}}$ stand for the duality product between $\mathscr{E}_{\Phi}$ and $(\mathscr{E}_{\Phi})'$ and between $(\mathcal{E}_{\Phi})'$ and $\mathcal{E}_{\Phi}$, respectively.

\medskip

\noindent The proof of Theorem \ref{c1theo0} and \ref{c1theo1} follows by the results in subsection \ref{subsecnec} and \ref{subsecsuff}.

\subsection{Necessity}\label{subsecnec}

We start by the necessity of the conditions i)-iii) in \eqref{c1eq10}-\eqref{c1eq40} and conditions i)-iii) in \eqref{c1eq1}-\eqref{c1eq4} that we state in two separate lemmas below, providing with more details the proof of the second, the first one being similar, and following as a corollary. 
\begin{lem}\label{c2lem0}
	Let $\Phi$ be a Young function of class $\Delta_{2}\cap \nabla_2$ and let $\lambda \in L^{\infty}_{\omega}\left(\Omega; \mathcal{M}(\mathbb{R}^{d\times N})\right)$ be such that $\{\lambda_{x}\}_{x\in \Omega}$ is a $W^{1}L^{\Phi}$-gradient Young measure. Then  
	\begin{itemize}
		\item[i)] there exists $u \in W^{1}L^{\Phi}(\Omega; \mathbb{R}^{d})$ 
		such that \eqref{c1eq10} holds;
		\item[ii)] for every $f \in \mathscr{E}_{\Phi}$, \eqref{c1eq20} holds 
		with  $\mathcal Qf$ being the quasiconvexification of $f$ in \eqref{Qfdef}.
		\item[iii)] \begin{equation*}
			x \mapsto \int_{\mathbb{R}^{d\times N}} \Phi(|\xi|) d\lambda_x(\xi) \in L^{1}(\Omega).
		\end{equation*}
	\end{itemize} 
\end{lem}

\begin{lem}\label{c2lem2}
	Let $\Phi$ be a Young function of class $\Delta_{2}\cap \nabla_2$ and let $\nu \in L^{\infty}_{\omega}\left(\Omega\times Y; \mathcal{M}(\mathbb{R}^{d\times N})\right)$ be such that $\{\nu_{(x,y)}\}_{(x,y)\in \Omega\times Y}$ is a two-scale  $W^{1}L^{\Phi}$-gradient Young measure. Then  
	\begin{itemize}
		\item[i)] there exist $u \in W^{1}L^{\Phi}(\Omega; \mathbb{R}^{d})$ and $u_{1} \in L^{1}(\Omega; W^{1}_{\#}L^{\Phi}_{per}(Y;\mathbb R^d))$ with $\left(u_1, \frac{\partial u_{1}}{\partial y_{i}}\right) \in L^{\Phi}(\Omega\times Y_{per};\mathbb R^{2d})$ ($1 \leq i \leq N$) such that \eqref{c1eq1} holds;
		\item[ii)] for every $f \in \mathcal{E}_{\Phi}$, \eqref{c1eq2} holds 
		with  $f_{hom}$ given by \eqref{c1eq3}
		\item[iii)] \begin{equation*}
		(x, y) \mapsto \int_{\mathbb{R}^{d\times N}} \Phi(|\xi|) d\nu_{(x,y)}(\xi) \in L^{1}(\Omega\times Y).
		\end{equation*}
	\end{itemize} 
\end{lem}
\begin{proof}
	Let $\{\nu_{(x,y)}\}_{(x,y)\in \Omega\times Y}$ be a two-scale gradient Young measure. 
	\par We start by showing that i) holds. 
	 By Definition \ref{c2def3} there exists $u\in W^{1}L^{\Phi}(\Omega; \mathbb{R}^{d})$ such that for every sequence $\{\varepsilon_{n}\}\to 0$ one can find $\{u_{n}\}\subset W^{1}L^{\Phi}(\Omega; \mathbb{R}^{d})$ such that $\{(\langle\cdot/\varepsilon_{n}\rangle, \nabla u_{n})\}$ generates the Young measure $\{\nu_{(x,y)}\otimes dy\}_{x\in \Omega}$ and $u_{n}\,\rightharpoonup\,u$ in $W^{1}L^{\Phi}(\Omega; \mathbb{R}^{d})$. 
	\par Up to a subsequence (still denoted by $u_{n}$), we can also assume that $\{\Phi(|\nabla u_{n}|)\}$ is equi-integrable (see \cite{KoZa2017}) and that there exists a function $u_{1} \in L^{1}(\Omega; W^{1}_{\#}L^{\Phi}_{per}(Y;\mathbb R^d))$ with $\left(u_1,\frac{\partial u_{1}}{\partial y_{i}}\right) \in L^{\Phi}(\Omega\times Y_{per}; \mathbb R^{2d})$ ($1 \leq i \leq N$) such that the sequence $\{\nabla u_{n}\}$ weakly two-scale converges to $\nabla u + \nabla_{y}u_{1}$ (see e.g. \cite[Theorem 4.2]{tacha1} and \cite[Remark 2]{tacha2}).  
	\par Consequently, for all $\phi \in \mathcal{C}_{c}^{\infty}(\Omega\times Y; \mathbb{R}^{d\times N})$ we have that 
		\begin{align}\label{c2eq7}
\displaystyle{	\lim_{n\to +\infty}	\int_{\Omega} \nabla u_{n}\cdot \phi\left(x, \langle\dfrac{x}{\varepsilon_{n}}\rangle\right) dx} \\
	 =\displaystyle{ \int_{\Omega}\int_{Y} \left(\nabla u(x)+ \nabla_{y}u_{1}(x,y) \right)\cdot \phi(x,y)  dy dx.  }\nonumber
	\end{align} 
	\par Set $f(x,y,\xi) = \xi\cdot\phi(x,y)$ for $(x,y,\xi) \in \Omega\times Y\times\mathbb{R}^{d\times N}$. As $f$ is a Carath\'{e}odory integrand (measurable in $x$ and continuous in $(y,\xi)$ and the sequence $\{f(\cdot, \langle\cdot/\varepsilon_{n}\rangle, \nabla u_{n}(\cdot))\}$ is equi-integrable, by Proposition \ref{ball1}, see v)) we get that 
		\begin{align}\label{c2eq8}
\displaystyle{\lim_{n\to +\infty}	\int_{\Omega} \nabla u_{n}\cdot \phi\left(x, \langle\dfrac{x}{\varepsilon_{n}}\rangle\right) dx }\\
	 = \displaystyle{\int_{\Omega}\int_{Y}\int_{\mathbb{R}^{d\times N}} \xi\cdot \phi(x,y) d\nu_{(x,y)}(\xi) dy dx.  } \nonumber
	\end{align}  
	 Consequently, from \eqref{c2eq7}-\eqref{c2eq8} we get for a.e. $(x,y)\in\Omega\times Y$
	\begin{equation*}
	\int_{\mathbb{R}^{d\times N}} \xi d\nu_{(x,y)}(\xi) = \nabla u(x) + \nabla_{y}u_{1}(x,y) 
	\end{equation*}
	which proves i).
	\par Let us see now that iii) is satisfied. As $\{\Phi(|\nabla u_{n}|) \}$ is equi-integrable then by Proposition \ref{ball1} v) we get that 
	\begin{equation*}
	 \int_{\Omega}\int_{Y}\int_{\mathbb{R}^{d\times N}} \Phi(|\xi|) d\nu_{(x,y)}(\xi) dy dx = \lim_{n\to +\infty}	\int_{\Omega} \Phi(|\nabla u_{n}|) dx < +\infty,
	\end{equation*} 
	which completes the proof of iii).
	\par Finally, condition ii) follows along the lines of \cite[Proof of Theorem 1.1]{baba1} by application of a $\Gamma$-convergence results  for homogenization of integral functionals in $W^{1}L^{\Phi}$ spaces.
	Since $f\in\mathcal{E}_{\Phi}$, it satisfies the $\Phi$-growth condition \eqref{c2eq6}, but it is not necessarily $\Phi$-coercive, but it is possible to approximate it  by a family of $\Phi$-coercive functions, i.e. for every $\alpha>0$ and $M>0$, define
	\begin{equation*}
	f_{M,\alpha}(y,\xi) := f_{M}(y,\xi) + \alpha\Phi(|\xi|)
	\end{equation*}
	  where $f_{M}(y,\xi) = \max\{-M,f(y,\xi)\}$. Then 
	  \begin{equation*}
	  \alpha\Phi(|\xi|)-M \leqslant f_{M,\alpha}(y,\xi) \leqslant (c+\alpha)(1+\Phi(|\xi|)),
	  \end{equation*}
	  for all $(y,\xi)\in \overline{\Omega}\times \mathbb{R}^{d\times N}$, with $c$ a positive constant. 
	  Then by \cite[ Theorem 1.1 ]{FTGNZ} ($\Gamma$-$\liminf$ inequality), we could first approximate $f$ by $f_{M,\alpha} \geqslant f$, as $\alpha \to 0$, and then reproduce the same chain of inequalities as in  \cite[Lemma 3.1 (iii)]{baba1}. Finally, exploiting the equi-integrability of $\Phi(|\nabla u_n|)$  we get 
	 the equi-integrability of $\{f(\langle\cdot/\varepsilon_{n}\rangle, \nabla u_{n})\}$  and the following chain of inequalities follows by
	 Proposition \ref{ball1} v)
	  \begin{align*}
	  	 \int_{A}\int_{Y}\int_{\mathbb{R}^{d\times N}} f(y,\xi) d\nu_{(x,y)}(\xi) dy dx=\\
	  \liminf_{n\to +\infty}	\int_{A} f\left(\langle\dfrac{x}{\varepsilon_{n}}\rangle, \nabla u_{n}(x)\right) dx  \geqslant  \int_{A} f_{\textup{hom}}(\nabla u(x)) dx. \nonumber
	  \end{align*}
	  A localization argument concludes the proof of ii) .
\end{proof}

\begin{proof}[Sketch of the proof of Lemma \ref{c2lem0}]
	The proofs of \eqref{c1eq10} and \eqref{c1eq40} follows averaging in $Y$ \eqref{c1eq1} and \eqref{c1eq4}, taking into account \eqref{tsYMaveYM}.
	Finally the proof of \eqref{c1eq20} is a consequence of \eqref{c1eq2} taking into account that $\mathscr E_\Phi\subset \mathcal E_\Phi$ and that $f_{\rm hom}$ in \eqref{c1eq3} reduces to $\mathcal Q f$ in \eqref{Qfdef} when $f$ does not depend on $x$.
	\end{proof}
	
	\begin{rem}Clearly the result could be proven also via a direct argument, observing that given a $W^{1}L^{\Phi}$-gradient Young measure $\lambda$ there exists $\{u_n\}\subset W^{1}L^{\Phi}(\Omega;\mathbb R^d)$ which generates it and weakly converges to $u$ in $W^{1}L^{\Phi}(\Omega;\mathbb R^d)$. Then \eqref{c2eq7} and \eqref{c2eq8} with $\phi \in C^\infty_c(\Omega,\mathbb R^d)$, together with the $\Phi$- equi-integrability gives the desired equality in \eqref{c1eq10}. The same arguments adopted to prove \eqref{c1eq4} ensure the validity of \eqref{c1eq40}.  Finally the proof of \eqref{c1eq2} could be obtained directly, replacing the $\Gamma$-convergence result in \cite{FTGNZ} by the relaxation results \cite[Theorems 4.7 and 4.11]{MM}. Indeed, under $\Delta_2\cap \nabla_2$ conditions on the function $\Phi$, the latter one satisifes the Sobolev-type property, the Rellich-type property and maximal (and density) property in \cite{MM}. Indeed the Sobolev-tye property is proven in \cite[Theorem 3.2]{DT}, the maximal property is proven in \cite{KoZa2017} and the compact embeddings are recalled in \cite{Elvira3} and the references therein. 
		
		For a relaxation result in the homogeneous case we also refer to the relaxation results that can be deduced by the dimensional reduction argument in \cite{LN}.
		
		\end{rem}
\subsection{Sufficiency}\label{subsecsuff}
Next we  show that conditions i)-iii) in \eqref{c1eq10}-\eqref{c1eq40} and in \eqref{c1eq1}-\eqref{c1eq4} are also sufficient to characterize $W^{1}L^{\Phi}$-gradient Young measures and (two-scale) $W^{1}L^{\Phi}$-gradient Young measures generated by bounded sequences in Orlicz-Sobolev spaces. As in \cite{kinder1, kinder2, baba1}, we can first consider the homogeneous case and get the result via a suitable application of Hahn-Banach Separation Theorem. The non-homogeneous ones will be obtained by approximation of two-scale $W^{1}L^\Phi$-gradient Young measures by piecewise constant ones.

\subsubsection{Homogeneous case}\label{sect3.2.1}

Our main here is to prove the following result.

\begin{lem}\label{c3lem10}
	Let $\Phi\in \Delta_{2}\cap \nabla_2$ be a Young function, $F\in \mathbb{R}^{d\times N}$ and $\lambda\in \mathcal{M}(\mathbb{R}^{d\times N})$ be such that $\nu\in \mathcal{P}(\mathbb{R}^{d\times N})$. Assume that 
	\begin{align*}
		F = \int_{\mathbb{R}^{d\times N}} \xi d\lambda(\xi),
	\\
	\mathcal Q	f(F) \leqslant \int_{\mathbb{R}^{d\times N}} f(\xi) d\lambda(\xi)
	\hbox{ for every }f\in \mathscr{E}_{\Phi}, \hbox{ and }\\
		\int_{\mathbb{R}^{d\times N}} \Phi(|\xi|) d\lambda(\xi)< +\infty.
	\end{align*}
	Then $\lambda$ is a homogeneous $W^{1}L^{\Phi}$-gradient Young measure.
\end{lem}
\begin{lem}\label{c3lem1}
Let $\Phi\in \Delta_{2}\cap\nabla_2$ be a Young function, $F\in \mathbb{R}^{d\times N}$ and $\nu\in L^{\infty}_{\omega}(Y;\mathcal{M}(\mathbb{R}^{d\times N}))$ be such that $\nu_{y}\in \mathcal{P}(\mathbb{R}^{d\times N})$ for a.e. $y\in Y$. Assume that 
\begin{align*}
F = \int_{Y}\int_{\mathbb{R}^{d\times N}} \xi d\nu_{y}(\xi) dy,
\\
f_{\textup{hom}}(F) \leqslant \int_{Y}\int_{\mathbb{R}^{d\times N}} f(y,\xi) d\nu_{y}(\xi) dy
\hbox{ for every }f\in \mathcal{E}_{\Phi}, \hbox{ and }
\\
 \int_{Y}\int_{\mathbb{R}^{d\times N}} \Phi(|\xi|) d\nu_{y}(\xi) dy < +\infty.
\end{align*}
Then $\{\nu_{y}\}_{y\in Y}$ is a homogeneous two-scale $W^{1}L^\Phi$-gradient Young measure.

\end{lem}

We start by presenting the main steps which lead to the proof of Lemma \ref{c3lem1}, since the arguments necessary to prove Lemma \ref{c3lem10} are similar and simpler. We start by introducing some technical tools and providing some auxiliary lemmas. 

\begin{rem}\label{c3rem1}
We observe that, by the same argument of \cite[Remark 3.3]{baba1}, one can prove that the sets $M_{F}$ and $M'_F$, defined, for $F\in \mathbb{R}^{d\times N}$, as
\begin{equation}\label{c3eq4}
	\begin{array}{l}
		M_{F} := \bigg\{\nu\in L^{\infty}_{\omega}(Y;\mathcal{M}(\mathbb{R}^{d\times N}))\,:\,\{\nu_{y}\}_{y\in Y}\,\textup{is\,a\,homogeneous\,two-scale}  \\
		\qquad\quad\quad W^{1}L^{\Phi}-\textup{gradient\,Young\,measure\;and}\; \int_{Y}\int_{\mathbb{R}^{d\times N}} \xi d\nu_{y}(\xi) dy = F  \bigg \},
	\end{array}
\end{equation} 
and 
\begin{equation*}
	\begin{array}{l}
		M'_{F} := \bigg\{\lambda\in \mathcal{M}(\mathbb{R}^{d\times N}))\,:\lambda \textup{ is \, a \, homogeneous}  \\
		\qquad\quad\quad W^{1}L^{\Phi}-\textup{gradient\,Young\,measure\;and}\; \int_{\mathbb{R}^{d\times N}} \xi d\lambda(\xi) = F  \bigg \},
	\end{array}
	\end{equation*}
	are independent of $\Omega$, i.e. it suffices to replace the sequences used in the proofs of \cite[Remark 3.3]{baba1}, which are elements of standard Sobolev spaces by sequences in Sobolev-Orlicz spaces.
\end{rem}
The next technical results are the Orlicz counterpart of \cite[Lemma 3.4]{baba1} and \cite[Lemma 3.2]{kinder1}, i.e. they allow us to construct (two-scale) $W^{1}L^{\Phi}$ gradient Young measures from measures in $M_F$ and $M'_F$, respectively, that are defined on disjoint subsets of $\Omega$. The proof is omitted, since it suffices to replace sequences in $W^{1,p}$- spaces in \cite{baba1} and \cite{kinder1} by Orlicz-Sobolev ones. On the other hand the results are crucial in the proof to get the convexity of the sets $M_{F}$ and $M'_F$.
\begin{lem}\label{c3lem3}
Let $D$ be an open subset of $\Omega$ with Lipschitz boundary, and let $\mu,\,\nu \in L^{\infty}_{\omega}(\Omega\times Y;\mathcal{M}(\mathbb{R}^{d\times N}))$ be such that $\{\mu_{(x,y)}\}_{(x,y)\in\Omega\times Y}$ and $\{\nu_{(x,y)}\}_{(x,y)\in\Omega\times Y}$ are two-scale $W^{1}L^{\Phi}$-gradient Young measures with same underlying deformation $u\in W^{1}L^{\Phi}(\Omega; \mathbb{R}^{d})$. Let 
\begin{equation*}
\sigma_{(x,y)} := \left\{\begin{array}{ll}
\mu_{(x,y)} & \textup{if}\;\, (x,y)\in D\times Y \\
\nu_{(x,y)} & \textup{if}\;\, (x,y)\in (\Omega\backslash D)\times Y.
\end{array} \right.
\end{equation*} 	
Then $\sigma \in L^{\infty}_{\omega}(\Omega\times Y;\mathcal{M}(\mathbb{R}^{d\times N}))$ and $\{\sigma_{(x,y)}\}_{(x,y)\in\Omega\times Y}$ is a two-scale $W^{1}L^{\Phi}$-gradient Young measure with underlying deformation $u\in W^{1}L^{\Phi}(\Omega; \mathbb{R}^{d})$.
\end{lem}

\begin{lem}\label{c3lem3}
	Let $D$ be an open subset of $\Omega$ with Lipschitz boundary, and let $\lambda,\,\theta \in L^{\infty}_{\omega}(\Omega ;\mathcal{M}(\mathbb{R}^{d\times N}))$ be such that $\{\lambda_{x}\}_{x\in\Omega}$ and $\{\theta_{x}\}_{x\in\Omega}$ are  $W^{1}L^{\Phi}$-gradient Young measures with same underlying deformation $u\in W^{1}L^{\Phi}(\Omega; \mathbb{R}^{d})$. Let 
	\begin{equation*}
		\kappa_{x} := \left\{\begin{array}{ll}
			\lambda_{x} & \textup{if}\;\, x\in D \\
			\nu_{x} & \textup{if}\;\, x\in \Omega\backslash D.
		\end{array} \right.
	\end{equation*} 	
	Then $\kappa \in L^{\infty}_{\omega}(\Omega;\mathcal{M}(\mathbb{R}^{d\times N}))$ and $\{\kappa_{x}\}_{x\in\Omega}$ is a $W^{1}L^{\Phi}$-gradient Young measure with underlying deformation $u\in W^{1}L^{\Phi}(\Omega; \mathbb{R}^{d})$.
\end{lem}
\begin{lem}\label{c3lem4}
Let $\Phi\in \Delta_{2} \cap \nabla_2$. Then the sets $M_{F}$ and $M'_F$ are convex and weak$^{\ast}$-closed subset of $(\mathcal{E}_{\Phi})'$ and $(\mathscr E_{\Phi})'$, respectively.	
\end{lem}
\begin{proof}
	For the reader's convenience we only write some details about $M_F$, highlighting the parts which differ from the proof of \cite[Lemma 3.5]{baba1}. The proof regarding $M'_F$ is even simpler and identical to \cite[Proposition 3.3]{kinder1}.
	
	Every element $\nu\in M_{F}$ can be identified with a homogeneous Young measure $\nu_{y}\otimes dy$. 
	\par First we prove that $M_{F}$ is a subset of $(\mathcal{E}_{\Phi})'$. For this purpose let $\nu\in M_{F}$. Arguing as in proof of Lemma 3.5 in \cite{baba1}, one can show that 
	\begin{equation*}
	K := \int_{Y}\int_{\mathbb{R}^{d\times N}} \Phi(|\xi|) d\nu_{y}(\xi) dy < +\infty.
	\end{equation*}
	Hence, using the fact that $\nu_{y}$ are probability measures for a.e. $y\in Y$, for every $f\in \mathcal{E}_{\Phi}$ we have that 
	\begin{align*}
	\displaystyle{\int_{Y}\int_{\mathbb{R}^{d\times N}} f(y,\xi) d\nu_{y}(\xi) dy  }&\leqslant \displaystyle{\|f\|_{\mathcal{E}_{\Phi}}  \int_{Y}\int_{\mathbb{R}^{d\times N}}  (1 + \Phi(|\xi|)) d\nu_{y}(\xi) dy}   \\
	 & =\displaystyle{ (1+K)\|f\|_{\mathcal{E}_{\Phi}}}. 
	\end{align*}
	As a consequence, $M_{F} \subset (\mathcal{E}_{\Phi})'$. \\
	Let us now prove that $M_{F}$ is closed for the weak$^{\ast}$-topology of $(\mathcal{E}_{\Phi})'$. Denoting by $\overline{M_{F}}$ the closure of $M_{F}$ for the weak$^{\ast}$-topology of $(\mathcal{E}_{\Phi})'$ it is enough show that $\overline{M_{F}} \subset M_{F}$. Since $\mathcal{E}_{\Phi}$ is separable, the weak$^{\ast}$-topology of $(\mathcal{E}_{\Phi})'$ is locally metrizable and thus, if $\nu \in \overline{M_{F}}$, there exists a sequence $\{\nu^{k}\} \subset M_{F}$ such that $\nu^{k}\,\rightharpoonup\,\nu$ in $(\mathcal{E}_{\Phi})'$. Hence, since the map $(y,\xi)\mapsto \xi_{ij}$ is in $\mathcal{E}_{\Phi}$ (where $1\leqslant i\leqslant d$ and $1\leqslant j\leqslant N$), we get, from the definition of weak$^{\ast}$-convergence in $(\mathcal{E}_{\Phi})'$, that 
	\begin{equation}\label{c3eq5}
	\int_{Y}\int_{\mathbb{R}^{d\times N}} \xi d\nu_{y}(\xi) dy = \lim_{k\to +\infty} \int_{Y}\int_{\mathbb{R}^{d\times N}} \xi d\nu_{y}^{k}(\xi) dy = F.
	\end{equation}
	It remains to show that $\{\nu_{y}\}_{y\in Y}$ is a homogeneous two-scale Young measure. By definition, given $\{\varepsilon_{n}\}\to 0$, for each $k\in\mathbb{N}$ there exist sequences $\{u_{n}^{k}\}_{n\in\mathbb{N}} \subset W^{1}L^{\Phi}(Y; \mathbb{R}^{d})$ such that  $\{(\langle\cdot/\varepsilon_{n}\rangle, \nabla u_{n}^{k})\}_{n\in\mathbb{N}}$ generate the homogeneous Young measures  $\nu_{y}^{k}\otimes dy$. For every $(z,\varphi)$ in a countable dense subset of $L^{1}(Y)\times\mathcal{C}_{0}(\mathbb{R}^{N}\times \mathbb{R}^{d\times N})$ we have that 
	\begin{align*}
&\displaystyle{	\lim_{k\to +\infty} \lim_{n\to +\infty}	\int_{Y} z(x) \varphi\left(\langle\dfrac{x}{\varepsilon_{n}}\rangle, \nabla u_{n}^{k}(x)\right) dx }\\
	 & =  	\displaystyle{\lim_{k\to +\infty}  \int_{Y}\int_{Y}\int_{\mathbb{R}^{d\times N}} z(x) \varphi(y,\xi) d\nu_{y}^{k}(\xi) dy dx }\\
	 & = \displaystyle{\int_{Y} z(x) dx \int_{Y}\int_{\mathbb{R}^{d\times N}}  \varphi(y,\xi) d\nu_{y}(\xi) dy,}
	\end{align*}  
	where we have used the fact that $\mathcal{C}_{0}(\mathbb{R}^{N}\times \mathbb{R}^{d\times N}) \subset \mathcal{E}_{\Phi}$ in the second equality. By a diagonalization argument we can find a sequence $\{k(n)\}\nearrow +\infty$ such that, setting $v_{n} := u_{n}^{k(n)}$, we have that 
	\begin{equation*}
	 \lim_{n\to +\infty}	\int_{Y} z(x) \varphi\left(\langle\dfrac{x}{\varepsilon_{n}}\rangle, \nabla v_{n}(x)\right) dx =  \int_{\Omega} z(x) dx \int_{Y}\int_{\mathbb{R}^{d\times N}}  \varphi(y,\xi) d\nu_{y}(\xi) dy.
	\end{equation*}  
	Thus, $\{\nu_{y}\}_{y\in Y}$ is a homogeneous two-scale Young measure, which together with \eqref{c3eq5} implies that $\nu\in M_{F}$. \\
	The proof of the convexity of $M_{F}$ is identical to the one of \cite[Lemma 3.5]{baba1}, thus it is omitted. 
\end{proof}
The sufficiency of conditions i)-iii) in \eqref{c1eq10}-\eqref{c1eq40} and in \eqref{c1eq1}-\eqref{c1eq4}, i.e. the proof of Lemmas \ref{c3lem10} and \ref{c3lem1} in the homogeneous case are identical to the proof of \cite[Lemma 3.2]{baba1}, just replacing in the first case $\mathcal E_p$ and $(\mathcal E_p)'$ therein by $\mathscr E_{\Phi}$ and $(\mathscr E_{\Phi})'$, for the first result and $\mathcal E_\Phi$ and $(\mathcal E_{\Phi})'$ for the second one.

We conclude this subsection observing that  the following result, whose proof is identical to the analogous results in \cite{baba1} and \cite{kinder1} allows to construct homogeneous (two-scale) $W^{1}L^{\Phi}$- gradient Young measures starting from general ones.
\begin{prop}\label{c3prop1}
	Let $\nu \in L^{\infty}_{\omega}(\Omega\times Y;\mathcal{M}(\mathbb{R}^{d\times N}))$ be such that $\{\nu_{(x,y)}\}_{(x,y)\in \Omega\times Y}$ is a two-scale $W^{1}L^{\Phi}$- gradient Young measure. Then for a.e. $a\in\Omega$, $\{\nu_{(a,y)}\}_{y\in Y}$ is a homogeneous two-scale $W^{1}L^{\Phi}$- gradient Young measure.
		Let $\lambda \in L^{\infty}_{\omega}(\Omega;\mathcal{M}(\mathbb{R}^{d\times N}))$ be such that $\{\lambda_{x}\}_{x\in \Omega}$ is a $W^{1}L^{\Phi}$- gradient Young measure. Then for a.e. $a\in\Omega$, $\lambda_a$ is a homogeneous $W^{1}L^{\Phi}$- gradient Young measure.
\end{prop}

\subsubsection{The nonhomogeneous case}
The general cases,  stated in Lemmas \ref{c3lem20} and \ref{c3lem2}, exactly as in \cite{kinder1} and \cite{baba1}, are based on Proposition \ref{c3prop1} and on a suitable decomposition of the domain $\Omega$, exploiting a variant of  Vitali's covering Theorem and a suitable approximation of (two-scale) $W^{1}L^{\Phi}$-gradient Young measures by (two-scale) $W^{1}L^{\Phi}$-gradient Young measures that are piecewise constant with respect to $x$.

\begin{lem}\label{c3lem2}
	Let $\Omega$ be a bounded and open subset of $\mathbb{R}^{N}$ with Lipschitz boundary.	Let $\Phi$ be a Young function of class $\Delta_{2}\cap \nabla_2$ and let $\lambda\in L^{\infty}_{\omega}\left(\Omega; \right. $ $\left.\mathcal{M}(\mathbb{R}^{d\times N})\right)$ be such that $\lambda_x\in \mathcal{P}(\mathbb{R}^{d\times N})$ for a.e. $x \in \Omega$. Suppose that  
	\begin{itemize}
		\item[i)] there exist $u \in W^{1}L^{\Phi}(\Omega; \mathbb{R}^{d})$ 
		satisfying \eqref{c1eq10}, 
		\item[ii)] for every $f \in \mathscr{E}_{\Phi}$,  \eqref{c1eq20} holds and 
		\item[iii)] \begin{equation*}
			x\mapsto \int_{\mathbb{R}^{d\times N}} \Phi(|\xi|) d\lambda_{x}(\xi) \in L^{1}(\Omega).
		\end{equation*}
	\end{itemize} 
	Then $\{\lambda_{x}\}_{x\in\Omega}$ is a two-scale gradient Young measure with underlying deformation $u$.
\end{lem}

\begin{lem}\label{c3lem20}
	Let $\Omega$ be a bounded and open subset of $\mathbb{R}^{N}$ with Lipschitz boundary.	Let $\Phi$ be a Young function of class $\Delta_{2}\cap \nabla_2$ and let $\nu \in L^{\infty}_{\omega}\left(\Omega\times Y;\right. $ $\left. \mathcal{M}(\mathbb{R}^{d\times N})\right)$ be such that $\nu_{(x,y)}\in \mathcal{P}(\mathbb{R}^{d\times N})$ for a.e. $(x,y) \in \Omega\times Y$. Suppose that  
	\begin{itemize}
		\item[i)] there exist $u \in W^{1}L^{\Phi}(\Omega; \mathbb{R}^{d})$ and $u_{1} \in L^{1}(\Omega; W^{1}_{\#}L^{\Phi}_{per}(Y;\mathbb R^d))$ with $\left(u_1,\frac{\partial u_{1}}{\partial y_{i}}\right) \in L^{\Phi}(\Omega\times Y_{per};\mathbb R^{2d})$ ($1 \leq i \leq N$) satisfying \eqref{c1eq1}, 
		\item[ii)] for every $f \in \mathcal{E}_{\Phi}$,  \eqref{c1eq2} holds and 
		\item[iii)] \begin{equation*}
			(x, y) \mapsto \int_{\mathbb{R}^{d\times N}} \Phi(|\xi|) d\nu_{(x,y)}(\xi) \in L^{1}(\Omega\times Y).
		\end{equation*}
	\end{itemize} 
	Then $\{\nu_{(x,y)}\}_{(x,y)\in\Omega\times Y}$ is a two-scale gradient Young measure with underlying deformation $u$.
\end{lem}

Both proofs are omitted.
They rely on Proposition \ref{c3prop1}, \cite[Lemma 7.9]{pedregal2} Remark \ref{c3rem1} and Lemma \ref{c2lem1}, first considering the case $u=0$ and then passing to the general case.
The next corollary asserts the independence of the sequence in Definition \ref{c2eq1}.
\begin{cor}\label{c3cor1}
	Let $\{u_{n}\}$ be a bounded sequence in $W^{1}L^{\Phi}(\Omega;\mathbb{R}^{d})$. Assume that there exists a sequence $\{\varepsilon_{n}\}\to 0$ such that the pair $\{(\langle\cdot/\varepsilon_{n}\rangle, \nabla u_{n})\}$ generates a Young measure $\{\nu_{(x,y)}\otimes dy\}_{x\in \Omega}$. Then the family $\{\nu_{(x,y)}\}_{(x,y)\in\Omega\times Y}$ is a two-scale gradient Young measure.
\end{cor}


\section{Proof of Theorem \ref{c1theo2}}\label{sect4}

Before proving Theorem \ref{c1theo2} we start by recalling Valadier's notion of \textit{admissible integrand} (see \cite{vala2}).
\begin{defn}\label{c4def1}
	A function $f : \Omega\times Y\times\mathbb{R}^{d\times N} \to [0,+\infty)$ is said to be an admissible integrand if for any $\eta>0$, there exist compact sets $K_{\eta}\subset\Omega$ and $V_{\eta}\subset Y$, with $\mathcal{L}^{N}(\Omega\backslash K_{\eta}) <\eta$ and $\mathcal{L}^{N}(Y\backslash V_{\eta}) <\eta$ and such that $f|_{K_{\eta}\times V_{\eta}\times\mathbb{R}^{d\times N}}$ is continuous.
\end{defn}  
We observe that from Lemma 4.11 in Barchiesi \cite{barchie2}, if $f$ is an admissible integrand then, for fixed $\varepsilon>0$, the function $(x,\xi)\mapsto f(x,\langle x/\varepsilon\rangle,\xi)$ is $\mathcal{L}(\Omega)\times \mathcal{B}(\mathbb{R}^{d\times N})$-measurable, where $\mathcal{L}(\Omega)$ and $\mathcal{B}(\mathbb{R}^{d\times N})$ denote, respectively, the $\sigma$-algebra of Lebesgue measurable subsets of $\Omega$ and Borel subsets of $\mathbb{R}^{d\times N}$. In particular, the functional \eqref{intro1} is well defined in $W^{1}L^{\Phi}(\Omega;\mathbb{R}^{d})$.
\begin{proof}(\textit{of Theorem \ref{c1theo2}})
	Let $u\in W^{1}L^{\Phi}(\Omega;\mathbb{R}^{d})$ and let $\{\varepsilon_{n}\} \to 0$. We start by showing that 
	\begin{equation}\label{c4eq1}
	\Gamma\textup{-}\limsup_{n\to +\infty} \mathcal{F}_{\varepsilon_{n}}(u) \leqslant \inf_{\nu\in \mathcal{M}_{u}} \int_{\Omega}\int_{Y}\int_{\mathbb{R}^{d\times N}} f(x,y,\xi) d\nu_{(x,y)}(\xi) dy dx,
	\end{equation}
	where $\mathcal{M}_{u}$ is the set defined in \eqref{c1mu}. Let $\nu\in \mathcal{M}_{u}$, by Remark \ref{c2rem2} there exists a sequence $\{u_{n}\}\subset W^{1}L^{\Phi}(\Omega;\mathbb{R}^{d})$  such that  $\{(\langle\cdot/\varepsilon_{n}\rangle, \nabla u_{n})\}$ generates the Young measure $\{\nu_{(x,y)}\otimes dy\}_{x\in \Omega}$ and $u_{n}\,\rightharpoonup\,u$ in $W^{1}L^{\Phi}(\Omega;\mathbb{R}^{d})$. Extract a subsequence $\{\varepsilon_{n_{k}}\} \subset \{\varepsilon_{n}\}$ such that 
	\begin{equation*}
	\limsup_{n\to +\infty} \mathcal{F}_{\varepsilon_{n}}(u_{n}) = \limsup_{k\to +\infty} \mathcal{F}_{\varepsilon_{n_{k}}}(u_{n_{k}})
	\end{equation*}
	and that $\{\Phi(|\nabla u_{n_{k}}|)\}$ is equi-integrable, which is always possible by the so-called Decomposition Lemma in the Orlicz-Sobolev setting (see \cite{KoZa2017}). In particular, due to 
	 \eqref{c1intro2}, the sequence $\{f(\cdot,\langle\cdot/\varepsilon_{n_{k}}\rangle, \nabla u_{n_{k}})\}$ is equi-integrable as well and applying Theorem 2.8 (ii) in Barchiesi \cite{barchie2} we get that 
	\begin{eqnarray}
	\Gamma\textup{-}\limsup_{n\to +\infty} \mathcal{F}_{\varepsilon_{n}}(u) &\leqslant&  \limsup_{n\to +\infty} \int_{\Omega} f\left(x, \langle\tfrac{x}{\varepsilon_{n_{k}}}\rangle,\nabla u_{n_{k}}(x)\right) dx \label{c4eq2} \\
	& = &  \int_{\Omega}\int_{Y}\int_{\mathbb{R}^{d\times N}} f(x,y,\xi) d\nu_{(x,y)}(\xi) dy dx. \nonumber
	\end{eqnarray}
	
	Taking the infimum over all $\nu\in \mathcal{M}_{u}$ in the right hand side of the \eqref{c4eq2} yields to \eqref{c4eq1}.
	\par Let us prove now that 
	\begin{equation}\label{c4eq4}
	\Gamma\textup{-}\liminf_{n\to +\infty} \mathcal{F}_{\varepsilon_{n}}(u) \geqslant \inf_{\nu\in \mathcal{M}_{u}} \int_{\Omega}\int_{Y}\int_{\mathbb{R}^{d\times N}} f(x,y,\xi) d\nu_{(x,y)}(\xi) dy dx.
	\end{equation}
	Let $\eta>0$ and $\{u_{n}\}\subset W^{1}L^{\Phi}(\Omega;\mathbb{R}^{d})$ such that $u_{n}\,\rightharpoonup\,u$ in $W^{1}L^{\Phi}(\Omega;\mathbb{R}^{d})$ and 
	\begin{equation}\label{c4eq5}
	\liminf_{n\to +\infty} \mathcal{F}_{\varepsilon_{n}}(u_{n}) \leqslant \Gamma\textup{-}\liminf_{n\to +\infty} \mathcal{F}_{\varepsilon_{n}}(u)+\eta.
	\end{equation}
	For a subsequence $\{n_{k}\}$, we can assume that there exists $\nu \in L^{\infty}_{\omega}\left(\Omega\times Y; \right. $ $\left. \mathcal{M}(\mathbb{R}^{d\times N})\right)$ such that $\{(\langle\cdot/\varepsilon_{n_{k}}\rangle, \nabla u_{n_{k}})\}$ generates a Young measure $\{\nu_{(x,y)}\otimes dy\}_{x\in \Omega}$ and 
	\begin{equation}\label{c4eq6}
\lim_{k\to +\infty} \mathcal{F}_{\varepsilon_{n_{k}}}(u_{n_{k}}) = \liminf_{n\to +\infty} \mathcal{F}_{\varepsilon_{n}}(u_{n}).
	\end{equation}
	We remark taht $\{\nabla u_{n_{k}}\}$ is equi-integrable since it is bounded in $L^{\Phi}(\Omega;\mathbb{R}^{d\times N})$. Thus, by Proposition \ref{ball1} (v) we get that for every $A\in\mathcal{A}(\Omega)$,
	\begin{equation*}
	\begin{array}{rcl}
	\int_{A} \nabla u(x) dx =  \lim_{k\to +\infty} \int_{A} \nabla u_{n_{k}}(x) dx 
	  =  \int_{A}\int_{Y}\int_{\mathbb{R}^{d\times N}} \xi d\nu_{(x,y)}(\xi) dy dx.
	\end{array}
	\end{equation*}
	By the arbitrariness of the set $A$, it follows that 
	\begin{equation}\label{c4eq7}
	\nabla u(x) = \int_{Y}\int_{\mathbb{R}^{d\times N}} \xi d\nu_{(x,y)}(\xi) dy \quad \textup{a.e.\;in}\; \Omega.
	\end{equation}
	As a consequence of Corollary \ref{c3cor1}, $\{\nu_{(x,y)}\}_{(x,y)\in\Omega\times Y}$ is a two-scale $W^1L^{\Phi}$-gradient Young measure and, by \eqref{c4eq7}, we also have that $\nu\in\mathcal{M}_{u}$. Applying now Theorem 2.8 (i) in Barchiesi \cite{barchie2} we get that 
	\begin{align*}
&	\displaystyle{\lim_{n\to +\infty} \int_{\Omega} f\left(x, \langle\dfrac{x}{\varepsilon_{n_{k}}}\rangle,\nabla u_{n_{k}}(x)\right) dx     }\\
&  \geqslant\displaystyle{  \int_{\Omega}\int_{Y}\int_{\mathbb{R}^{d\times N}} f(x,y,\xi) d\nu_{(x,y)}(\xi) dy dx }\\
& \geqslant \displaystyle{\inf_{\nu\in \mathcal{M}_{u}} \int_{\Omega}\int_{Y}\int_{\mathbb{R}^{d\times N}} f(x,y,\xi) d\nu_{(x,y)}(\xi) dy dx.}
	\end{align*}
	Hence by \eqref{c4eq5}, \eqref{c4eq6} and the arbitrariness of $\eta$ we get the desired result. Gathering \eqref{c4eq1} and \eqref{c4eq4}, we obtain that 
	\begin{equation*}
	\Gamma\textup{-}\lim_{n\to +\infty} \mathcal{F}_{\varepsilon_{n}}(u) = \inf_{\nu\in \mathcal{M}_{u}} \int_{\Omega}\int_{Y}\int_{\mathbb{R}^{d\times N}} f(x,y,\xi) d\nu_{(x,y)}(\xi) dy dx.
	\end{equation*}
	It remains to prove that the minimum is attained. To this aim, consider a recovery sequence $\{\bar{u}_{n}\}\subset W^{1}L^{\Phi}(\Omega;\mathbb{R}^{d})$. Arguing exactly as before we can assume that (a subsequence of) $\{\nabla \bar{u}_{n}\}$ generates a two-scale gradient Young measure $\{\nu_{(x,y)}\}_{(x,y)\in \Omega\times Y}$ that $\nu\in\mathcal{M}_{u}$ and  $\{f(\cdot,\langle\cdot/\varepsilon_{n}\rangle, \nabla \bar{u}_{n})\}$ is equi-integrable. According to Theorem 2.8 (ii) in Barchiesi \cite{barchie2} and using the fact that $\{\bar{u}_{n}\}$ is a recovery sequence,
	\begin{align*}
	\Gamma\textup{-}\lim_{n\to +\infty} \mathcal{F}_{\varepsilon_{n}} &= 
	\lim_{n\to +\infty} \int_{\Omega} f\left(x, \langle\dfrac{x}{\varepsilon_{n}}\rangle,\nabla \bar{u}_{n}(x)\right) dx     \\
	& =    \int_{\Omega}\int_{Y}\int_{\mathbb{R}^{d\times N}} f(x,y,\xi) d\nu_{(x,y)}(\xi) dy dx
	\end{align*}
	which completes the proof.
\end{proof}
Let us conclude by stating a corollary which provides an alternative formula to derive the homogenized energy density $f_{\textup{hom}}$ in \eqref{c1eq3}.
\begin{cor}
	If $f : Y\times\mathbb{R}^{d\times N}\to [0,+\infty)$ is a Carath\'{e}odory integrand (independent of $x$) and satisfying \eqref{c1intro2}, then for every $u\in W^{1}L^{\Phi}(\Omega;\mathbb{R}^{d})$,
	\begin{equation*}
	\mathcal{F}_{\textup{hom}}(u) = \int_{\Omega} f_{\textup{hom}}(\nabla u(x)) dx,
	\end{equation*}
where for every $F\in\mathbb{R}^{d\times N}$,
\begin{equation*}
f_{\textup{hom}}(F)= \min_{\nu \in M_{F}} \int_{Y}\int_{\mathbb{R}^{d\times N}} f(y,\xi) d\nu_{y}(\xi) dy
\end{equation*}	
and $M_{F}$ is defined in \eqref{c3eq4}.
\end{cor}
\begin{proof}
	It is proven in \cite{FTGNZ} that
	\begin{equation*}
	\mathcal{F}_{\textup{hom}}(u) = \int_{\Omega} f_{\textup{hom}}(\nabla u(x)) dx
	\end{equation*} 
	where $f_{\textup{hom}}$ is defined in \eqref{c1eq3}. By Theorem \ref{c1theo2} with $\Omega=Y$ and $u(x)=Fx$, we get that 
	\begin{equation*}
	f_{\textup{hom}}(F)= \min_{\nu \in \mathcal{M}_{u}} \int_{Y}\int_{Y}\int_{\mathbb{R}^{d\times N}} f(x,y,\xi) d\nu_{(x,y)}(\xi) dy dx.
	\end{equation*}
	The thesis follows from Proposition \ref{c2lem3}.
\end{proof}
\begin{rem}
	Let $\Phi(t) = \frac{t^{p}}{p}$ ($p>1, \;t\geq 0$), then $\Phi \in \Delta_{2}\cap \nabla_2$ (with $\widetilde{\Phi}(t)=\frac{t^{q}}{q}$, $q=\frac{p}{p-1}$) and one has $L^{\Phi}(\Omega;\mathbb{R}^{d})\equiv L^{p}(\Omega;\mathbb{R}^{d})$, $L^{\Phi}(\Omega\times Y_{per}) \equiv L^{p}(\Omega; L^{p}_{per}(Y))$ and $W^{1}L^{\Phi}(\Omega;\mathbb{R}^{d})\equiv W^{1,p}(\Omega;\mathbb{R}^{d})$.  Therefore, the integral functional in \eqref{intro1} can be written as
	\begin{equation*}
	\mathcal{F}_{\varepsilon}(u) := \int_{\Omega} f\left(x,\langle\frac{x}{\varepsilon}\rangle, \nabla u(x)\right) dx,
	\end{equation*}  
	with $u\in W^{1,p}(\Omega;\mathbb{R}^{d})$. Thus we find the same results of homogenization of Theorem \ref{c1theo1} and Theorem \ref{c1theo2} in classical Sobolev spaces (see \cite{baba1}).
\end{rem}

%

\vspace{0.3cm}

\textbf{Acknowledgements.}  Fotso Tachago is grateful to  Department of Basic and Applied Science for Engineering of Sapienza - University of Rome for its kind hospitality, during the preparation of this work. He also acknowledges the support received by International Mathematical Union, through IMU grant 2024. E.~Zappale acknowledges the support of the project
``Mathematical Modelling of Heterogeneous Systems (MMHS)",
financed by the European Union - Next Generation EU,
CUP B53D23009360006, Project Code 2022MKB7MM, PNRR M4.C2.1.1.  She is a member of the Gruppo Nazionale per l'Analisi Matematica, la Probabilit\`a e le loro Applicazioni (GNAMPA) of the Istituto Nazionale di Alta Matematica ``F.~Severi'' (INdAM). 
She also acknowledges partial funding from the GNAMPA Project 2023 \emph{Prospettive nelle scienze dei materiali: modelli variazionali, analisi asintotica e omogeneizzazione}.



\begin{thebibliography}{99}
	
	
	\bibitem{allair1} \textsc{Allaire G.}, \textit{Homogenization and two scale convergence}, SIAM, J. Math. Anal., \textbf{23}, (1992), 1482-1518.

\bibitem{baba1} \textsc{Babadjian J-F., Baia M. and  Santos P. M.}, \textit{Characterization of two-scale gradient Young measures and application to homogenization}, Applied Mathematics and Optimization, \textbf{57} (2007), 69-97.

\bibitem{bia} \textsc{Baia M. and Fonseca I.}, \textit{$\Gamma$-convergence of functionals with periodic integrands via 2-scale convergence}, 

\bibitem{ball1} \textsc{Ball J.M.}, \textit{A version of the fundamental theorem for Young measures, PDE's and continum models for phase transitions}, Lecture notes physics, \textbf{334} (1989), 207-215.
	


\bibitem{barchie2} \textsc{Barchiesi M.}, \textit{Loss of polyconvexity by homogenization: a new example},  Calc. Var. and Partial Diff. Eq., \textbf{30} (2007), 215-230.


\bibitem{b}\textsc{Bianconi B.},
\textit{A new proof of semicontinuity by {Y}oung measures and an
approximation theorem in {O}rlicz-{S}obolev spaces},
 {Abstr. Appl. Anal.}, {\bf 15}, (2003), 881--898.





\bibitem{CDDEA1}
\textsc{Cioranescu D., Damlamian A. and De Arcangelis
	R.}, \textit{Homogenization of nonlinear integrals via the periodic
	unfolding method}, {C. R. Math. Acad. Sci. Paris},
{\bf 339}, (2004), n.1, 77--82.

\bibitem{CDDEA2} \textsc{Cioranescu D., Damlamian A. and De Arcangelis
		R.},
\textit{Homogenization of quasiconvex integrals via the periodic
		unfolding method}, {SIAM J. Math. Anal.},
{\bf 37}, (2006), n. 5, 1435--1453.




\bibitem{CDG1}
	\textsc{Cioranescu D., Damlamian A. and Griso G.},
	\textit{Periodic unfolding and homogenization},
	{C. R. Math. Acad. Sci. Paris},
{\bf 335}, (2002), n. 1, 99--104.

\bibitem{CDG2}\textsc{Cioranescu D., Damlamian A. and Griso G.},
\textit{The periodic unfolding method in homogenization},
{SIAM J. Math. Anal.},
{\bf 40}, (2008), n. 4, 1585--1620.

\bibitem{CDGbook}
\textsc{Cioranescu D., Damlamian A. and Griso G.},
\textit{The periodic unfolding method},
{Series in Contemporary Mathematics}, {\bf 3},
{Theory and applications to partial differential problems},
{Springer, Singapore}, (2018), xv+513.


\bibitem{DG}\textsc{Desch W. and Grimmer R.},
\textit{On the wellposedness of constitutive laws involving
dissipation potentials},
{Trans. Amer. Math. Soc.},
{\bf 353}, (2001), n. 12, 5095--5120.

	\bibitem{diper1} \textsc{Diperna R. J.}, \textit{Convergence of approximate solutions to conservation laws}, Arch. Rational Mech. Anal. \textbf{82} (1983), 27-70.
	
	\bibitem{DT} \textsc{Donaldson K.T. and Trudinger N.S.},\textit{Orlicz-Sobolev spaces and Imbedding thoerems}, J. Funct. Anal., {bf 8}, (1971), 52-75.

	
	\bibitem{w} \textsc{E W.}, \textit{Homogenization of linear and nonlinear transport equations}, Comm. Pure Appl. Math., \textbf{45} (1992), 301-326.
	
\bibitem{FLbook} \textsc{Fonseca I. and Leoni G.} \textit{Modern Methods in the Calculus of Variations: $L^p$ spaces}. Springer Monographs in Mathematics. Springer, New York, 2007.
	
	\bibitem{fonse2} \textsc{Fonseca I., M\"{u}ller S. and Pedegral P.}, \textit{Analysis of concentration and oscillation effects generated by gradients}, SIAM J. Math. Anal. \textbf{29} (1998), 736-756.
	
	\bibitem{FMAq} \textsc{Fonseca I. and M\"{u}ller S.},
	 \textit{{$\mathscr A$}-quasiconvexity, lower semicontinuity, and {Y}oung
		measures}, {SIAM J. Math. Anal.}, {\bf 30}, n.6, (1999), 1355--1390.
	\bibitem{tacha1} \textsc{Fotso Tachago J. and Nnang H.}, \textit{two scale convergence of integral functional with convex periodic and nonstandard growth integrands}, Acta Appl. Math., \textbf{121}, (2012), 175-196.
	

	
	\bibitem{tacha2} \textsc{Fotso Tachago J., Gargiulo G., Nnang H. and Zappale E.}, \textit{Multiscale Homogenization of integral convex functionals in Orlicz-Sobolev Setting}
	{Evolution Equations and Control Theory}, {\bf 10}, n.2,(2021), 297 – 320.

	\bibitem{tacha5} \textsc{Fotso Tachago J., Gargiulo G., Nnang H. and Zappale E.}, \textit{Some convergence results on the periodic nfolding operator in Orlicz setting}, Integral Methods in Science and Engineering, Conference Proceedings, (2023), 361-371.
	
	\bibitem{FTGNZ} \textsc{Fotso Tachago J., Gargiulo G., Nnang H. and  Zappale E.}, \textit{Homogenization of non-convex integral energies with Orlicz growth, via periodic unfolding},  submitted.

	\bibitem{tacha6} \textsc{Fotso Tachago J., Nnang H. and Zappale E.}, \textit{Relaxation of Periodic and Nonstandard Growth Integrals by means of Two-scale convergence}, Integral Methods in Science and Engineering., (2019), DOI: $10.1007/978-030-16077-7_{-}10$ 
	
	\bibitem{tacha3} \textsc{Fotso Tachago J., Nnang H. and Zappale E.}, \textit{Reiterated periodic homogenization of integral functional with convex and nonstandard growth integrands}, Opuscula Math., \textbf{41}, (2021), 113-143.
	

	
	\bibitem{FTNZDIE}\textsc{Fotso Tachago J., Nnang H. and Zappale E.},\textit{Reiterated Homogenization of nonlinear degenerate elliptic operators with nonstandard growth},  Differential Integral Equations {\bf 37} (9/10), (2024), 717-752, DOI: 10.57262/die037-0910-717B.

	
	


\bibitem{kinder1} \textsc{Kinderlehrer D. and Pedegral P.}, \textit{Gradient Young measures generated by sequences in Sobolev spaces}, J. Geom. Anal., \textbf{4} (1988), 59-90.

\bibitem{kinder2} \textsc{Kinderlehrer D. and Pedegral P.}, \textit{Characterization of Young measures generated by gradients}, Arch. Rational Mech. Anal., \textbf{115} (1991), 329-365.


\bibitem{KoZa2017} \textsc{Kozarzewski P. A. and Zappale E.},
	\textit{Orlicz equi-integrability for scaled gradients},
{Journal of Elliptic and Parabolic Equations}, {\bf 3}	(2017), {1-2}, 1 – 13.

\bibitem{Elvira3} \textsc{Kozarzewski P. A. and  Zappale E.}, \emph{A note on Optimal
	design for thin structures in the Orlicz-Sobolev setting,} Proceedings of the IMSE conference 2016, DOI
10.1007/978-3-319-59384-5-14.

\bibitem{LN}\textsc{Laskowski W. and Nguyen H. T.},
\textit{Effective energy integral functionals for thin films in the
	{O}rlicz-{S}obolev space setting},
{Demonstratio Math.}, {\bf 46}, n. 3
(2013), {585--604}.
\bibitem{LNW}\textsc{Lukkassen D., Nguetseng G. and Wall, P.},\textit{ Two scale convergence}, Int. J. Pure Appl. Math. \textbf{2},(2002), 35–86.




\bibitem{MM} \textsc{Mingione G. and Mucci D.},
\textit{Integral functionals and the gap problem: sharp bounds for
relaxation and energy concentration},
{SIAM J. Math. Anal.},
\textbf{36}, (2005), n. 5
1540--1579.
	\bibitem{mignon2} \textsc{Mingione G. and Radulescu V.}, \textit{Recent developments in problems with nonstandard growth and nonuniform ellipticity}, J. Math. Anal. Appl. \textbf{501} (2021) 125197.
	
	
	\bibitem{nguet1} \textsc{Nguetseng G.}, \textit{A general convergence result for a functional related to the theory of homogenization}, SIAM Journal on Mathematical Analysis \textbf{20}(3), 1989, 608-623.


\bibitem{pedregal2} \textsc{Pedregal P.}, \textit{Parametrized measures and variational principles}, Progress in Nonlinear Differential Equations and their Applications, \textbf{30}, Birkh\"{a}user Verlag, Basel (1997).

\bibitem{pedregal1} \textsc{Pedregal P.}, \textit{Vector variational problems and applications to optimal design}, ESAIM Control Optim. Calc. Var., \textbf{11} (2005), 357-381.
\bibitem{pedregal3} \textsc{Pedregal P.}, \textit{Multiscale Young measures}, Trans. Am. Math. Soc., \textbf{358} (2006), 591-602.



\bibitem{vala1} \textsc{Valadier M.}, \textit{D\'{e}sint\'{e}gration d'une mesure sur un produit}, C. R. Acad Ac. Paris \textbf{276} S\'{e}rie A (1973), 33-35.
\bibitem{vala2} \textsc{Valadier M.}, \textit{Admissible functions in two-scale convergence}, Portugaliae Mathematica, \textbf{54} (1997), 148-164.



	
\end{thebibliography}
\end{document}